\documentclass[12pt]{extarticle}
\usepackage[utf8]{inputenc}
\usepackage[T1]{fontenc}

\usepackage{caption}
\oddsidemargin \evensidemargin
\usepackage{setspace}
\usepackage{amssymb}
\usepackage{amsmath}
\usepackage{comment}
\usepackage{amsthm}
\usepackage{amssymb}
\usepackage{algorithmicx}
\usepackage{algorithm}
\usepackage{algpseudocode}
\usepackage{mathtools}
\usepackage{nicematrix}
\usepackage[margin=1cm]{caption}
\usepackage{subcaption}
\usepackage{enumerate}
\usepackage{enumitem}
\setlist[enumerate]{itemsep=-2pt, topsep=2pt}
\usepackage{multirow}
\usepackage{tikz,tikz-3dplot,tikz-cd}
\usetikzlibrary{decorations.markings, arrows.meta} 
\tikzset{
  midarrow/.style={
    postaction={
      decorate,
      decoration={
        markings,
        mark=at position 0.5 with {\arrow{Stealth}}
      }
    }
  }
}
\usepackage[dvipsnames]{xcolor}
\usepackage{pgfplots}
\usepackage{url}
\usepackage{csquotes}
\usepackage[backend=biber,
style=trad-alpha,doi=false]{biblatex} 
\usepackage{hyperref}
\hypersetup{
    colorlinks=true,
    linkcolor=orange!80!black,
    filecolor=magenta,      
    urlcolor=blue,
    citecolor=RoyalBlue,
}

\definecolor{mycolor1}{rgb}{0.00000,0.44700,0.74100}
\definecolor{mycolor2}{rgb}{0.8500, 0.3250, 0.0980}
\definecolor{mycolor3}{rgb}{0.9290, 0.6940, 0.1250}
\definecolor{mycolor4}{rgb}{0.4940, 0.1840, 0.5560}
\definecolor{mycolor5}{rgb}{0.4660, 0.6740, 0.1880}

\theoremstyle{definition}
\newtheorem{definition}{Definition}[section]
\theoremstyle{theorem}

\newtheorem{proposition}[definition]{Proposition}
\newtheorem{corollary}[definition]{Corollary}
\newtheorem{theorem}[definition]{Theorem}

\newtheorem{lemma}[definition]{Lemma}

\newtheorem*{theorem*}{Theorem}

\newtheorem{mainthm}{Theorem}

\newtheoremstyle{italic}
  {.5\topsep}{.5\topsep}{\normalfont}{}
  {\itshape}{.}{.5em}
  {\thmname{#1}\thmnumber{ \textit{#2}}\thmnote{ \normalfont(#3)}}

\theoremstyle{italic}
\newtheorem{remark}[definition]{Remark}
\newtheorem{construction}[definition]{Construction}
\newtheorem{examplex}[definition]{Example}
\newenvironment{example}
  {\pushQED{\qed}\examplex}
  {\popQED\endexamplex}

\usepackage{mathtools} 
\mathtoolsset{showonlyrefs,showmanualtags}
\numberwithin{equation}{section}

\newcommand{\R}{\mathbb{R}}

\newcommand{\C}{\mathbb{C}}

\let\O\relax
\newcommand{\O}{\mathcal{O}}
\let\P\relax
\newcommand{\P}{\mathbb{P}}
\newcommand{\PP}{\mathbb{P}}
\newcommand{\A}{\mathbb{A}}

\newcommand{\cV}{\mathcal{V}}
\newcommand{\cL}{\mathcal{L}}

\newcommand{\mO}{\mathcal{O}}

\newcommand{\dsheaf}{\overset{\bullet}{\omega}}

\newcommand{\conv}{\operatorname{Conv}}
\newcommand{\wt}{\widetilde}

\DeclareMathOperator{\gr}{Gr}
\newcommand{\grk}{\operatorname{Gr}(k,k+2)}

\newcommand{\aA}{\mathcal{A}}

\newcommand{\crank}{\operatorname{cr}}
\newcommand{\sing}{\operatorname{Sing}}
\newcommand{\open}{\operatorname{int}}

\newcommand{\res}{\operatorname{Res}}

\newcommand{\grw}{\operatorname{gr}^W}
\newcommand{\grf}{\operatorname{gr}_F}
\newcommand{\olog}{\Omega_{\log}}
\newcommand{\oOmega}{\mathring{\Omega}}

\newcommand{\into}{\hookrightarrow}

\newcommand{\cF}{\mathcal{F}}
\newcommand{\cO}{\mathcal{O}}

\title{{The $m=2$ amplituhedron is a positive geometry}}
\author{Joris Koefler, Dmitrii Pavlov, Rainer Sinn}
\date{}

\begin{document}

\maketitle

\begin{abstract}
    Amplituhedra are semi-algebraic sets of fundamental importance in particle physics. Their discovery initiated the field of Positive Geometry.
    We settle a central conjecture in this field: we show that the $m=2$ amplituhedron is a positive geometry.
    Our main new tool is a Hodge-theoretic criterion for being a positive geometry, connecting the work of Brown--Dupont with the original setup of Arkani-Hamed--Bai--Lam.
\end{abstract}

\section{Introduction}\label{sec:intro}

The \emph{amplituhedron} $\mathcal{A}_{n,k,m}$, 
introduced in \cite{Amplituhedron}, is a semi-algebraic subset of the Grassmannian.
It plays an important role in particle physics, where it is used to calculate scattering amplitudes in $\mathcal{N}=4$ Super Yang--Mills quantum field theory. 
It is also a central object in Positive Geometry \cite{whatisposgeo}, a new field at the intersection of mathematics and theoretical physics. 
The key notion in this field is that of a \emph{positive geometry} \cite{pos_geom_2017}. This is a pair $(X,S)$ consisting of a real projective variety $X$ and a semi-algebraic subset $S$ such that there exists a unique rational differential top form on $X$ with simple poles along the boundary of $S$ satisfying a number of additional technical properties.
Known examples of positive geometries include polytopes, positive parts of toric varieties and the totally non-negative Grassmannian.
It is, however, difficult to prove that a given pair $(X,S)$ is a positive geometry in general.
In particular, a central conjecture \cite[Conjecture 6.3]{pos_geom_2017} that has remained open since the field was established is that amplituhedra are positive geometries.

\medskip

In this article we settle this conjecture for the $m=2$ amplituhedron.
The cases $m=2$ and $m=4$ are of immediate relevance for physics.
Our approach is motivated by an exciting recent development in the theory of positive geometries.
In \cite{brown-dupont} Brown and Dupont recast this notion in the language of Deligne's mixed Hodge theory, providing a  systematic framework to study positive geometries. In this framework, positive geometries arise from pairs of complex varieties for which the genus, a certain Hodge-theoretic invariant, vanishes. 
The new framework, however, is not equivalent to the original one proposed in \cite{pos_geom_2017}.
At the heart of the latter are semi-algebraic sets, whereas the former is purely about the complex geometry and topology of a pair of varieties.
In this article, we establish a conceptual link by providing a sufficient mixed Hodge-theoretic criterion for being a positive geometry. We then apply it to the $m=2$ amplituhedron.
The result can be phrased as follows, and the necessary terminology and notation are introduced in Sections \ref{sec:prelim} and \ref{sec:PG_MHS}.

\begin{mainthm} \label{thm:main_1}
        Let $S$ be a non-empty regular semi-algebraic subset of a smooth real projective variety $X$ of dimension $n$ with orientable interior.
    Let $D$ be the algebraic boundary of $S$.
    Suppose $S$ has a well-defined face stratification $\cF(S)$ which is a graded poset of rank $n$. Let $R:=\mathcal{R}(S)$ be the residual arrangement of $S$.
    Further, assume for every face $F$ of $S$ of dimension $d>0$ with Zariski closure $V$ the following.
    \begin{enumerate}
        \item $V$ is normal and $V\setminus W$ is smooth, where $W$ is the union of all varieties in the algebraic stratification $\cL(D)$ that are codimension two subsets of $V$;

        \item $H^d(V,\mathcal{O}_V)=0$;

        \item The interior of $F$ is contained in $V\setminus D_V$, where $D_V$ is the union of all varieties in $\mathcal{L}(D)$ that are codimension one subsets of $V$.
    \end{enumerate}
    Suppose $(X\setminus R,D\setminus R)$ is a recursive genus zero pair, and the set $R$ is disjoint from $S$.
    Then $(X,S)$ is a positive geometry with canonical form  $r_*\oOmega(\iota_*[S])$.
\end{mainthm}

The mathematical structure of amplituhedra has been studied extensively over the last decade.
The description of the $m=1$ amplituhedron was given in \cite{karpwilliams}, and in this case it is a linear object.
For the case of $m=2$ the boundary structure was described in \cite{lukowski} and \cite{lam_face_structure}. Tilings (i.e. subdivisions with special properties) of the $m=2$ amplituhedron are one of the central subjects of \cite{bao2019tilings}, \cite{LukowskiParisiWilliams2023}, and \cite{parisi:meq2}.
Tilings in the $m=4$ case are studied in \cite{BCFW2025}, \cite{m4tilings2023} and subsequent work.
Other interesting developments include, for instance, \cite{Galashin2024origamiI, Galashin2026origamiII,GalashinLam2020}, and the full body of work is too large to describe here.
Despite this vast amount of deep work, the central conjecture \cite[Conjecture 6.3]{pos_geom_2017} that amplituhedra are positive geometries has only been settled in a few special cases. 
\begin{enumerate}
    \item[$\bullet$] $k=1$: The amplituhedron is a cyclic polytope \cite{bernd1988}, and polytopes are positive geometries \cite[Section 6]{pos_geom_2017}. 
    \item[$\bullet$] $m=1$: The amplituhedron is the union of bounded cells of a cyclic hyperplane arrangement \cite{karpwilliams}, so again a polyhedral object. 
    \item[$\bullet$] $k+m=n$: In this case the amplituhedron is isomorphic to $\mathrm{Gr}(k,n)_{\geq 0}$, which gives rise to a positive geometry by \cite[Theorem 8]{lam2024invitation}.
    \item[$\bullet$] $k=m=2$: This is proven in \cite{amplituhedron_adjoint}.
\end{enumerate}

Our mixed Hodge-theoretic criterion in Theorem~\ref{thm:main_1}, combined with combinatorial results mentioned above, implies that the $m=2$ amplituhedron is a positive geometry.
This is, however, not at all straightforward. 
A natural way to go from a pair $(X,S)$ with a semi-algebraic set $S$ to a pair of varieties $(X,D)$ with a divisor $D\subset X$ is to set $D$ to be the algebraic boundary $\partial_a S$ of $S$, i.e. the Zariski closure of its boundary in the analytic topology on $X$. 
Setting $X:=\mathrm{Gr}(k,k+2)$ and $D:=\partial_a \mathcal{A}_{n,k,2}(Z)$, however, does not provide a genus zero pair. 
Indeed, the main result in \cite{KPScube} states that for $k\geq 3$ and large enough $n$ the pair $(X,D)$ has strictly positive genus. 
We will therefore have to alter this pair to reduce its genus in a way that locally preserves the geometry of the amplituhedron. It is not clear a priori that this is possible at all.
We now briefly comment on our approach. 

\medskip

As we show in \cite{KPScube}, the obvious obstacles for the pair $(\mathrm{Gr}(k,k+2),\mathcal{\partial}_a \mathcal{A}_{n,k,2}(Z))$ to have genus zero come from the \emph{residual arrangement} of the amplituhedron, see Definition \ref{def:faces_residual}.
These are strata of $\mathcal{\partial}_a \mathcal{A}_{n,k,2}(Z)$ that are disjoint from the semi-algebraic set $\mathcal{A}_{n,k,2}(Z)$.
In this article we will show that \emph{all} obstacles come from the residual arrangement, i.e. that each face of the amplituhedron is supported on a variety of genus zero (Proposition \ref{prop:faceshavegenuszero}).
We will then remove the residual arrangement from each of the varieties in the original pair, thus removing positive contributions to the genus.
In order to show that the new pair satisfies the assumptions of Theorem \ref{thm:main_1}, we will have to deal with subtleties of genera and differential forms on singular and open varieties.
After this, our main result follows from Theorem \ref{thm:main_1}.

\begin{mainthm} \label{thm:main_2}
        Let $Z$ be a generic $n\times (k+2)$ matrix with positive maximal minors.
    Then $(\gr(k,k+2), \mathcal{A}_{n,k,2}(Z))$ is a positive geometry. 
\end{mainthm}

This article is organized as follows.
In Section \ref{sec:prelim} we introduce the necessary notions and recall basic results that will be needed in later parts of the article.
In Section \ref{sec:PG_MHS} we recall the framework of \cite{brown-dupont}, explain how to make it applicable for non-compact varieties, and prove Theorem \ref{thm:main_1}. 
Section \ref{sec:faces} is devoted to the algebro-geometric properties of the faces of $\mathcal{A}_{n,k,2}(Z)$: we show their Zariski closures are unirational varieties with rational singularities by relying on the results of \cite{lam_face_structure}. We also show that any such variety (and any Zariski open subset thereof) has genus zero.
Finally, in Section~\ref{sec:last} we prove Theorem \ref{thm:main_2}.

\section{Preliminaries}\label{sec:prelim}

In this section we introduce the basic notions used in this article.
We recall the definitions of the totally non-negative Grassmannian, positroid varieties, the amplituhedron, and positive geometries.
We also define the algebraic boundary stratification of a codimension one subvariety and a semi-algebraic set. 

\medskip

Denote by $\gr(k,n)$ the Grassmannian of $k$-dimensional linear subspaces of $\C^n$.
We usually consider the Grassmannian in its Pl\"ucker embedding into $\P^{N}$ with $N=\binom{n}{k}-1$.
If it is clear from context, we will write $\Lambda$ both for the point in $\gr(k,n)$ and a full rank $k\times n$ matrix representing it.
The Pl\"ucker coordinates $p_I$ with $I\subset [n]\coloneqq \{1,\ldots, n\}$ of size $k$ are given by the $k\times k$ minors of $\Lambda$ indexed by the columns in $I$.

\begin{definition}[Totally non-negative Grassmannian]\label{def:non-negative_grass}
    The \emph{totally non-negative Grassmannian} $\gr(k,n)_{\geq 0}$ \cite{postnikov} is the semi-algebraic subset of $\gr(k,n)$, defined as the intersection of $\gr(k,n)$ in its Pl\"ucker embedding into $\PP^{N}$ with the non-negative orthant $\P^{N}_{\geq 0}$. 
Equivalently, $\gr(k,n)_{\geq 0}$ consists of all linear spaces represented by real matrices $\Lambda\in \R^{k\times n}$ of rank $k$ whose $k\times k$ minors are all non-negative, up to a global change of sign. 
\end{definition}

The totally non-negative Grassmannian admits a stratification into positroid cells, see, e.g., \cite{postnikov}.
The boundary structure of the $m=2$ amplituhedron, which is the main object of this article, is closely related to the positroid stratification \cite{lam_face_structure}. We thus now review the basics and fix notation.
For a subset $\mathcal{M}\subset \binom{[n]}{k}$ define
\[
  \Pi_{\mathcal{M}}^{> 0} := \{\Lambda \in \gr(k,n)_{\geq 0} \mid  p_I(\Lambda) > 0 \text{ for } I \in \mathcal{M},\; p_I(\Lambda)=0 \text{ for } I \notin \mathcal{M} \}.
\]
If  $\Pi_{\mathcal{M}}^{> 0} \neq \varnothing$, this set is called a \emph{positroid cell}. In this situation $\mathcal{M}$ is a matroid, called a \emph{positroid}. Replacing strict inequalities with non-strict ones, we get a space $\Pi^{\geq 0}_\mathcal{M}$.
The complex Zariski closure of $\Pi_{\mathcal{M}}^{> 0}$ in $\gr(k,n)$ is called a \emph{positroid variety}, and is denoted by
\[
  \Pi_{\mathcal{M}}:= \overline{\Pi_{\mathcal{M}}^{> 0}} \subset \gr(k,n).
\]
The quasi-projective \emph{open positroid varieties} $\mathring{\Pi}_{\mathcal{M}}$ are defined as 
\[
    \mathring{\Pi}_{\mathcal{M}} := \Pi_{\mathcal{M}}\setminus \left( \bigcup_{\mathcal{M}'\subset \mathcal{M}} \Pi_{\mathcal{M}'}\right).
\]
By seminal work of Postnikov \cite{postnikov}, the totally non-negative Grassmannian admits a decomposition
\[
    \gr(k,n)_{\geq 0} = \bigsqcup_{\mathcal{M}}{\Pi}^{>0}_{\mathcal{M}}
\]
into positroid cells, each of which is an open ball. Moreover, in \cite{psw} Postnikov, Speyer, and Williams prove that this decomposition equips $\gr(k,n)_{\geq 0}$ with the structure of a CW-complex.
We furthermore have $\mathring{\Pi}_{\mathcal{M}}\cap\gr(k,n)_{\geq 0} = \Pi^{> 0}_{\mathcal{M}}$.

\medskip

We now give a definition of the central object of this article.
\begin{definition}[Amplituhedron] \label{def:amplituhedron}
    Fix positive integers $n,k$, and $m$ with $k+m\leq n$. Fix an ${n\times (k+m)}$ matrix $Z$ with strictly positive maximal minors. This matrix defines the rational map $\widetilde{Z}: \gr(k,n)\dashrightarrow \gr(k,k+m)$ via $\Lambda \mapsto \Lambda Z$. The image $\widetilde{Z}(\mathrm{Gr}(k,n)_{\geq 0})$ is called the \emph{amplituhedron}~$\mathcal{A}_{n,k,m}(Z)$. 
\end{definition}
Due to the positivity assumption on $Z$, the map $\widetilde{Z}$ is regular on $\mathrm{Gr}(k,n)_{\geq 0}$, by the Cauchy--Binet formula for the maximal minors of $\Lambda Z$. So the amplituhedron is a compact semi-algebraic subset of $\gr(k,k+m)$.
Henceforth, we will often suppress the dependence on $Z$ and write $\aA_{n,k,m}$ for $\aA_{n,k,m}(Z)$.
Our next goal is to recall the definition of positive geometries in the sense of \cite{pos_geom_2017}, see also \cite{lam2024invitation}. We start with auxiliary definitions. 

\begin{definition}[Algebraic boundary]
    For a semi-algebraic set $S$ inside a real variety $X$, the \emph{algebraic boundary} of $S$ is the Zariski closure in $X$ of its boundary $\partial S$ in the analytic topology. The algebraic boundary of $S$ is denoted by $\partial_a S$. 
\end{definition}

\begin{definition}[Regular semi-algebraic sets]\label{def:regularSAset}
    Let $X$ be a real projective variety and $S$ a semi-algebraic subset.
    Then, we say $S$ is \emph{regular} if $S$ is equal to the closure of its interior in the analytic topology on $X$.
\end{definition}

We are now ready to recall the definition of positive geometries in the sense of \cite{pos_geom_2017}.
\begin{definition}[Positive geometries]\label{def:pos_geom_lam}
    Let $X$ be a real projective $d$-dimensional variety.
    Let $X_{\geq 0}$ be a regular semi-algebraic subset of $X$ with orientable interior.
    Suppose the algebraic boundary of $X_{\geq 0}$ decomposes into irreducible codimension one subvarieties $D_1,\ldots, D_r$ of $X$.
    Denote by $(D_i)_{\geq 0}$  the analytic closure of the interior of $D_i\cap \partial X_{\geq 0}$ in $D_i(\R)$.
    The pair $(X, X_{\geq 0})$ is called a \emph{positive geometry} if there exists a unique rational $d$-form $\Omega_{(X,X_{\geq 0})}$ on $X$, called the \emph{canonical form}, satisfying the following properties.
    \begin{enumerate}[label=(\alph*)]
        \item If $d > 0$, the form $\Omega_{(X,X_{\geq 0})}$ has only simple poles, and these are exactly along $D_1,\ldots,D_r$. 
        \item For each component $D_i$ of the algebraic boundary, the pair $(D_i,(D_i)_{\geq 0})$ is a positive geometry; 
         the residue $\operatorname{Res}_{D_i} \Omega_{(X, X_{\geq 0})}$ equals the canonical form $\Omega_{(D_i,(D_i)_{\geq 0})}$.
        \item If $d=0$, $X=X_{\geq 0}$ is a point and $\Omega_{(X, X_{\geq 0})} = \pm 1$. 
        The sign determines the orientation on this $0$-dimensional manifold.
    \end{enumerate}
\end{definition}

\begin{remark}
    The uniqueness of the canonical form is essentially equivalent to the statement that $X$ has no non-zero holomorphic top forms if $X$ is a normal variety: adding such a form to a canonical form would produce another canonical form, and the difference between two canonical forms is a holomorphic form, since they have the same poles and residues along these poles. See the proof of Theorem \ref{thm:fromBDtoPG} for more details.
\end{remark}

The amplituhedron is a top-dimensional regular semi-algebraic set in $\gr(k,k+m)$.
The central goal of this article is to prove that the $m=2$ amplituhedron is a positive geometry for arbitrary values of $k$ and $n$ and generic choices of $Z$.
Definition \ref{def:pos_geom_lam} suggests that to do this, it is essential to understand the algebraic boundary $\partial_a \mathcal{A}_{n,k,2}(Z)$. We now introduce the notion of the algebraic boundary stratification of a semi-algebraic set. 

\begin{definition}[Algebraic boundary stratification] \label{def:alg_bd_strat}
    Let $X$ be a complex variety and $Y$ a subvariety of $X$ of pure codimension one.
    We define the \emph{algebraic stratification} $\mathcal{L}(Y)$ of $Y$ as the smallest poset (ordered by inclusion) satisfying the following two properties:
    \begin{enumerate}
        \item $X$ and each irreducible component of $Y$ are elements of $\mathcal{L}(Y)$.
        \item For any finite number of elements in $\mathcal{L}(Y)$ with non-empty intersection, each irreducible component of this intersection is again an element of $\mathcal{L}(Y)$. 
    \end{enumerate}
    If $S$ is a regular semi-algebraic subset of a projective variety $X$, with algebraic boundary $\partial_aS$, then we define the \emph{algebraic boundary stratification} of $S$ as $\cL(S)\coloneqq \cL(\partial_aS)$. We call the elements of $\cL(S)$ \emph{boundary strata} of $S$. These objects are defined analogously inside a variety of the form $X\setminus R$, where $R$ is a closed subvariety of $X$ disjoint from $S$. 
\end{definition}

\begin{remark}
    If $S$ is a regular semi-algebraic subset of a smooth variety $X$, then its algebraic boundary $\partial_a S$ is empty or pure codimension $1$. This follows from standard results in semi-algebraic dimension theory along the lines of \cite[Lemma~2.5]{Sinn2015}.
\end{remark}

\begin{example}
    When $S$ is a polytope in $\mathbb{P}^n$, the algebraic boundary stratification $\mathcal{L}(S)$ is isomorphic to the poset of flats of the facet hyperplane arrangement of $S$.
\end{example}

\begin{example}\label{exm:bd_strat_gr(2,4)}
    The algebraic boundary stratification of $\gr(k,n)_{\geq 0}$ is its positroid stratification (and is isomorphic to the cell decomposition in \cite{postnikov}, see also \cite[Section 2.3]{williams2023positive}). 
    The case $\gr(2,4)_{\geq 0}$ already shows the importance of taking the irreducible components of the intersections of boundary divisors. 
    The algebraic boundary of $\gr(2,4)_{\geq 0}$ consists of four irreducible components, given inside $\gr(2,4)$ by the equations $p_{12}=0$, $p_{14}=0$, $p_{23}=0$ and $p_{34}=0$, where $p_{ij}$ are the Pl\"ucker coordinates.
    However, the intersection $\{p_{12}=p_{23}=0\}$ in $\gr(2,4)$ is reducible: it consists of two planes, each containing a positroid cell.
    Each of these two planes is an element of $\mathcal{L}(\gr(2,4)_{\geq 0})$.
\end{example}

We conclude this section by defining faces, the residual arrangement, corners, and vertices of a semi-algebraic set. 

\begin{definition}[Faces and the residual arrangement]\label{def:faces_residual}
    Let $S$ be a regular semi-algebraic subset of a real projective variety $X$.
    We define the {faces} of $S$ recursively in dimension as follows. For an irreducible component $D$ of $\partial_a S$, we define the \emph{facet} $S\vert_D$ of $S$ on the boundary stratum $D$ as the closure of the interior of $D\cap \partial S$.
    We continue recursively by setting $S\vert_{D'}$ to be the closure of the interior of $D'\cap \partial S\vert_D$ where $D'$ is an irreducible component of $\partial_a S\vert_D$. The sets obtained in this recursion are called the \emph{faces} of $S$.
    If the Zariski closures of the faces, called \emph{face strata}, are elements of the algebraic boundary stratification $\cL(S)$, then the collection of these strata is called \emph{face stratification} of $S$, denoted $\cF(S)$.
    By convention, the face stratification includes $X$.
    Any stratum in $\cL(S)$ that is not in $\mathcal{F}(S)$ is called \emph{residual}.  
    The \emph{residual arrangement} of $S$, denoted $\mathcal{R}(S)$, is the union of all residual strata.
\end{definition}

\begin{remark}
    Note that, by construction, all faces of a regular semi-algebraic set $S$ are regular semi-algebraic sets again.
\end{remark}

\begin{definition}[Corners and vertices]\label{def:corner_vertices}
   Let $S$ be a regular semi-algebraic subset of a smooth real projective variety $X$.
    A \emph{corner} of $S$ is any zero-dimensional boundary stratum in $\cL(S)$. 
    A \emph{vertex} of $S$ is a zero-dimensional face of $S$.
\end{definition}

\begin{remark}\label{rem:conrner_top_implies_corner_face}
    Suppose we have a full-dimensional semi-algebraic subset $S\subset X$ 
    such that its face stratification $\cF(S)$ is a graded poset of rank $n:=\dim X$.
    Then each face of $S$ has a vertex. 
    Indeed, the Zariski closure of every face is contained in a maximal chain of $\cF(S)$ and, since all such chains in  $\cF(S)$ have length $n+1$, must contain a zero-dimensional element.
\end{remark}

\begin{example}
    We continue with the totally non-negative Grassmannian $\gr(2,4)_{\geq 0}$ from Example \ref{exm:bd_strat_gr(2,4)}, where we saw that we only obtain all positroid strata when also allowing for intersecting the irreducible components of the intersections of the four algebraic boundary components.
    This allows us to find zero-dimensional strata in $\cL(\gr(2,4)_{\geq 0})$: there are six corners corresponding to the points in $\gr(2,4)$ representing the lines $\operatorname{span}_{\C}(e_i,e_j)$, where $e_i$ are the standard basis vectors of $\C^4$.
    All of these are vertices of $\gr(2,4)_{\geq 0}$ since their Pl\"ucker coordinates are all obviously non-negative.
\end{example}

Let us also demonstrate an example in which not all corners are vertices.

\begin{example} \label{ex:cornersverices}
    Consider the non-convex polygon $X_{\geq 0}$ in $\R^2$ given by the union of the two triangles $\conv\{(0,0),(-1,0),(0,2)\}$ and $\conv\{(0,0),(1,0),(0,1)\}$, see Figure \ref{fig:nonconvex-polygon}.
    \begin{figure}[htbp]
      \centering
      \begin{tikzpicture}[
        scale=2.5,
        line join=round,
        vertex/.style={circle, fill=black, inner sep=1.6pt},
        residual/.style={
          circle, draw=black, fill=white,
          inner sep=1.8pt, line width=0.7pt
        }
      ]
        \coordinate (A) at (-1,0);
        \coordinate (O) at (0,0);
        \coordinate (B) at (1,0);
        \coordinate (C) at (0,1);
        \coordinate (D) at (0,2);
        \coordinate (R) at ({-1/3},{4/3});
    
        \fill[gray!15] (A)--(B)--(C)--(D)--cycle;
    
        \draw[dashed, gray] (O)--(C);
        \draw[dashed, gray] (C)--(-0.65,1.65);
    
        \draw[thick] (A)--(B)--(C)--(D)--cycle;
    
        \node[vertex, label=below left:{$(-1,0)$}] at (A) {};
        \node[vertex, label=below right:{$(1,0)$}] at (B) {};
        \node[vertex, label=right:{$(0,1)$}] at (C) {};
        \node[vertex, label=above:{$(0,2)$}] at (D) {};
    
        \node[residual, label=below:{$(0,0)$}] at (O) {};
        \node[residual,
          label=left:{$\left(-\frac13,\frac43\right)$}
        ] at (R) {};
    
        \node at (-0.30,0.55) {$X_{\geq 0}$};
      \end{tikzpicture}
      \caption{The non-convex polygon $X_{\geq 0}$ and its six
        corners. The two residual corners are marked by open circles.}
      \label{fig:nonconvex-polygon}
    \end{figure}
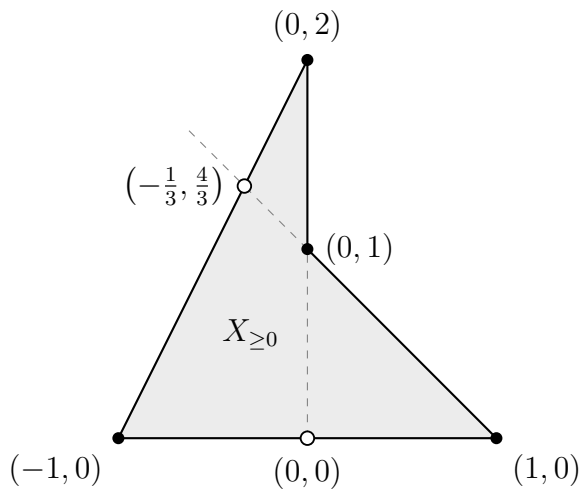
    This configuration has six corners, namely the five vertices of the two triangles ($(-1,0)$, $(0,0)$, $(1,0)$, $(0,1)$, $(0,2)$) and the intersection point  $(-1/3,4/3)$ of the two lines spanned by $(-1,0)$ and $(0,2)$, and $(0,1)$ and $(1,0)$, respectively. Among these six points, the four vertices of the non-convex polygon are the vertices according to our definition. The corner $(0,0)$, for example, is not a vertex, because it is an interior point of the one-dimensional face, which is the line segment connecting the points $(-1,0)$ and $(1,0)$. It is therefore a residual stratum.
    The numerator of the canonical form of this set is a linear polynomial that vanishes at the residual points, i.e., at~the corners that are not vertices.
\end{example}

\section{A mixed Hodge theory criterion for positive geometries}\label{sec:PG_MHS}

As mentioned in the introduction, the framework for positive geometries suggested by Brown and Dupont \cite{brown-dupont} based on mixed Hodge theory is not equivalent to that proposed originally in \cite{pos_geom_2017}.
Therefore, a natural goal is to find a clear criterion describing when the mixed Hodge theory approach gives rise to positive geometries as in Definition \ref{def:pos_geom_lam}.
In this section we make this comment more precise and give such a criterion, which is the main conceptual result of this article.

\subsection{Recalling the basics}
We start by recalling the basic notions used in \cite{brown-dupont}.
\begin{definition}\label{def:mhs}
    Let $H$ be a finite-dimensional $\mathbb{Q}$-vector space with a decreasing filtration $F^\bullet$ on the complexification $H_{\mathbb{C}}$ called \emph{Hodge filtration}, i.e.
    \[
        \ldots \subset   F^{i+1}H_\C \subset F^{i}H_\C \subset \ldots 
    \]
    We say $H$ is a \emph{pure Hodge structure of weight $k$} if there exists a decomposition \linebreak $H_\C=\bigoplus_{p+q=k}H^{p,q}_{\C}$ with $\overline{H^{q,p}}=H^{p,q}$, and such that 
    \[
        F^aH_\C=\bigoplus_{\substack{p\geq a\\p+q=k}}H^{p,q},
    \]
    for all integers $a$.
    Let $W_\bullet$ be an increasing filtration on $H$ called \emph{weight filtration}, i.e.
     \[
         \ldots \subset W_{i}H \subset W_{i-1}H \subset \ldots.  
     \]
    For any non-negative integer $p$ write
    \[
        \grf^p H_\C \coloneqq \frac{F^pH_\C}{F^{p+1}H_\C},\quad \text{and }\grw_w H := \frac{W_w H}{W_{w-1} H}.
    \]
    We say $H$ is a \emph{mixed Hodge structure} if for every integer $w$, the filtration induced by $F^\bullet$ on $(\grw_wH)_\C$ defines a pure Hodge structure of weight $w$.
    For $w=p+q$ one defines
    \[
        H^{p,q} := (\operatorname{gr}^W_w H)_{\C}^{p,q}=\grf^p(\grw_w H)_{\C}.
    \]
    The corresponding mixed \emph{Hodge numbers} are the complex vector space dimensions
    \[
        h^{p,q} = \dim H^{p,q} = h^{q,p}.
    \]
\end{definition}

\begin{remark}
    See \cite{Deligne1971} for the construction of these two filtrations.
    In our setting, we will exclusively deal with mixed Hodge structures on the $k$-th relative homology group with rational coefficients of a pair of complex varieties $(X,D)$, where $X$ is of dimension $n$ and $D$ is a closed subvariety of $X$.
    Then, for the Hodge numbers we explicitly write
    \begin{align}\label{eq:Hodge_numbers}
        h^{p,q}(H_k(X,D)) = \dim \operatorname{gr}_F^p\bigl(\operatorname{gr}^W_{p+q} H_{k}(X,D)\bigr)_{\C}.
    \end{align}
\end{remark}

The framework suggested in \cite{brown-dupont} relies on the introduction of two Hodge-theoretic invariants of the mixed Hodge structure of $H_n(X,D)$.
The first is the \emph{genus} $g(X,D)$ of the pair which is defined as the sum of the ``outer'' Hodge numbers
\[
    g(X,D) \coloneqq \sum\limits_{p>0} h^{-p,0}(H_n(X,D)).
\]
When $D=\varnothing$, we write $g(X):=g(X,\varnothing)$.

The second invariant is the \emph{combinatorial rank} which is defined as
\[
    \crank(X,D)\coloneqq h^{0,0}(H_n(X,D)).
\]
Similarly, if $D=\varnothing$, we write $\crank(X)\coloneqq \crank(X,\varnothing)$.
The main theorem in the work of Brown and Dupont states that for any pair $(X,D)$ of genus zero such that $X\setminus D$ is smooth and $X$ is compact, there is a well-defined map, essentially induced by Poincar\'e duality,
\begin{equation}\label{eq:BD}
    \Omega: H_n(X,D)\twoheadrightarrow \Omega_{\mathrm{log}}^n(X\setminus D)
\end{equation}
associating with each relative homology cycle $\sigma \in H_n(X,D)$ a holomorphic $n$-form on $X\setminus D$ with logarithmic poles along $D$ (the target space $\Omega_{\log}^n(X\setminus D)$ is the space of such forms).
This form is called the \emph{canonical form} of $\sigma$.
Canonical forms defined in this way satisfy a number of nice properties, including compatibility with taking residues, additivity and functoriality.
If $g(X,D)=0$ and if $\sigma \in H_n(X,D)$ is the homology class of a semi-algebraic set $X_{\geq 0}\subset X(\mathbb{R})$ with boundary contained in $D$ and connected interior such that $(X,X_{\geq 0})$ is a positive geometry with canonical form $\Omega_{(X,X_{\geq 0})}$, then by \cite[Lemma 2.12]{logcanonical} $\Omega(\sigma)=\Omega_{(X,X_{\geq 0})}$ is the canonical form of this positive geometry in the original sense of \cite{pos_geom_2017}.
In the most natural case we have $D=\partial_a X_{\geq 0}$.

\begin{remark}
    While Brown and Dupont define canonical forms, they do not give a definition of a positive geometry in the language of mixed Hodge theory. Such a definition is presented in \cite{sturmfels2026positivegeometriescubicsurfaces} and is based on the notion of \emph{positive arrangements}. This definition is not equivalent to the one in \cite{pos_geom_2017}, as the notion of positive arrangements is not recursive in nature. One can show that the $m=2$ amplituhedron is a positive geometry of the positive arrangement obtained from $(\mathrm{Gr}(k,k+2), \partial_a \mathcal{A}_{n,k,2}(Z))$ by blowing up all desingularized residual strata in increasing dimension. Such a sequence of blowups keeps the ambient variety smooth but introduces new singularities to the face strata of the amplituhedron. To show that the $m=2$ amplituhedron is a positive geometry in the original sense of \cite{pos_geom_2017} we will thus pursue a different strategy based on removing the residual arrangement rather than blowing it up. 
\end{remark}

We now comment on the inequivalence of the frameworks of \cite{pos_geom_2017} and \cite{brown-dupont}.
On the one hand, the framework of \cite{brown-dupont} is purely complex. 
There exist pairs $(X,D)$ of genus zero for which the variety $D$ does not have real points and thus cannot contain the boundary of any semi-algebraic set.
This is the case e.g. when $X$ is a rational curve and $D$ is a complex point on it. On the other hand, as we demonstrate in our previous article \cite[Section 6]{KPScube}, there exist semi-algebraic sets $X_{\geq 0}\subset X$ such that $(X,X_{\geq 0})$ is a positive geometry, but $X_{\geq 0}$ gives rise to a homology cycle in $H_n(X,D)$ with $g(X,D)>0$, where $D$ is the algebraic boundary of $X_{\geq 0}$.
However, in many important cases (for instance, for convex polytopes) these frameworks are in perfect agreement.

\medskip

In view of this, it is desirable to have a criterion in the language of mixed Hodge theory for being a positive geometry in the sense of \cite{pos_geom_2017}. The main goal of this section is to present such a criterion, see Theorem \ref{thm:fromBDtoPG}. 

\medskip

The first step towards Theorem \ref{thm:fromBDtoPG} is the definition of recursive genus zero pairs. Note that in the definition below the varieties are not assumed to be compact. 

\begin{definition}[Recursive genus zero pairs]\label{def:recursive_genus_zero} 
    Let $(X,D)$ be a pair of complex varieties with $X$ irreducible.
    Let $\cL(D)$ be the algebraic stratification of $D$.
    For each $W\in\cL(D)$, denote by $D_W\subset W$ the union of all codimension one subvarieties of $W$ that are elements of $\cL(D)$.
    We say $(X,D)$ is a \emph{recursive genus zero pair} if $g(W,D_W)=0$ for all $W\in \cL(D)$.
\end{definition}

\begin{remark}
    As demonstrated in \cite[\S 2.4.3]{brown-dupont}, having genus zero is not a recursive condition. Thus, not every genus zero pair is a recursive genus zero pair. 
\end{remark}

\begin{remark}
    By \cite[Proposition 3.9]{brown-dupont}, the genus of a pair is insensitive to codimension two subvarieties. Therefore, in Definition \ref{def:recursive_genus_zero} we can require $D_W$ to have codimension one in $W$ without any loss of generality. See Example \ref{ex:nonstrong_pair} for a semi-algebraic set giving rise to a recursive genus zero pair whose face stratification is not graded. This example will elucidate the importance of the notion of graded face stratifications in the context of positive geometries.
\end{remark}

\subsection{Canonical forms from non-compact varieties}

The input of the framework of \cite{brown-dupont} is a pair of compact complex varieties $(X,D)$ of genus zero such that $X\setminus D$ is smooth. It is pointed out in the introduction to that article that the formalism of canonical forms should still hold for non-compact pairs. We would like to consider the amplituhedron $\mathcal{A}_{n,k,2}$ inside a pair of non-compact varieties to remedy the fact that the genus of the pair $(\mathrm{Gr}(k,k+2), \partial_a \mathcal{A}_{n,k,2})$ is in general non-zero. Thus, in this subsection we make this claim precise. This is done largely by adapting the statements in \cite{brown-dupont} to the non-compact case. 

\begin{remark}
    Canonical forms in the non-compact setting for hyperplane arrangements were studied in recent work \cite{anaelle}, where they are defined without relying on compactifying the pair. Our construction uses a compactification with special properties but is independent of the particular choice, see Lemma \ref{lem:open_canonical_map}. 
\end{remark}

We start by recalling a few more essential notions from \cite{brown-dupont}.
\begin{definition}[Modification of pairs]
    \label{def:modification}
    Let $(X,D)$ and $(X',D')$ be two pairs of complex varieties.
    A morphism of pairs $f:(X',D')\to(X,D)$ is a \emph{modification} if $f:X'\to X$ is proper, $D'=f^{-1}(D)$, and the restriction
    \[
        f|_{X'\setminus D'}:X'\setminus D'\to X\setminus D
    \]
    is an isomorphism.
\end{definition}

We also recall the notion of good compactifications.
\begin{definition}[Good compactification] \label{def:good_compactification}
    Let $(U,V)$ be a pair of complex varieties with $U\setminus V$ smooth.
    A \emph{good compactification} of $(U,V)$ consists of a compactification $\mathcal{X}$ of $U$ and the set $\mathcal{D}:=\overline{V}$, with the closure taken in $\mathcal{X}$, such that:
    \begin{enumerate}
        \item $\mathcal{X}\setminus\mathcal{D}$ is smooth;
        \item writing $\mathcal{B}=\mathcal{X}\setminus U$, the triple $(\mathcal{X},\mathcal{D},\mathcal{B})$ is locally of product type.
    \end{enumerate}
    The latter means that analytically locally the triple has the form
    \[
        (A\times C,\ D_A\times C,\ A\times B_C)
    \]
    for complex analytic pairs $(A,D_A)$ and $(C,B_C)$.
\end{definition}

In what follows, we will slightly abuse notation and for two sets $A$ and $B$ not contained in each other write $A\setminus B$ for $A\setminus (A\cap B)$.

\begin{remark}\label{rem:snc_good_compactification}
    Let $\mathcal{X}$ be smooth and compact, and let $\mathcal{B},\mathcal{D}\subset\mathcal{X}$ be divisors with no common component.
    If $\mathcal{B}\cup\mathcal{D}$ is a simple normal crossing divisor, then $(\mathcal{X},\mathcal{D})$ is a good compactification of $(\mathcal{X}\setminus\mathcal{B}, \mathcal{D}\setminus\mathcal{B})$.
\end{remark}

We fix the following setup for the rest of the section.
Let $X$ be a projective irreducible complex variety of dimension $n$, and let $D$ be a closed codimension one subvariety as usual.
Further, let $R$ be a closed subvariety in $X$, and $W$ a closed subvariety of $D$.
We assume that $X\setminus R$ is of dimension $n\geq 1$, and $D\setminus R$ is pure of dimension $n-1$.
Assume $W\subset D$ is closed, of dimension at most $n-2$, and such that $X\setminus(R\cup W)$ and $D\setminus(R\cup W)$ are smooth.
In this section we consider the mixed Hodge theory approach for the open pair $(X\setminus R,D\setminus R)$.
All subvarieties, inverse images, intersections and closures are taken with their reduced structures.
Homology and cohomology have rational coefficients.

\begin{construction}\label{const:good_cmp}
    Take $X$ as the compactification of $X\setminus R$.
    By Hironaka's embedded resolution theorem \cite{Hir64}, we can resolve singularities of $X$ in the following way.
    There is a morphism $\eta:T\to X$ from a smooth projective variety $T$ that is an isomorphism over the smooth locus of $X$ and in particular over $X\setminus(R\cup W)$, because this latter set is contained in the smooth locus of $X$.
    Set
    \[
        B \coloneqq T\setminus \eta^{-1}(X\setminus R),\quad \text{and }
        D'\coloneqq \overline{\eta^{-1}(D\setminus R)}, \quad \text{and } W'\coloneqq \overline{\eta^{-1}(W\setminus R)},
    \]
    where the closures are taken in $T$.
    Now, we use Hironaka's principalization theorem, see also \cite[Theorem 1.0.1]{Wlo05}, applied to the ideal sheaf given by the product of the ideal sheaves of $B$, $D'$, and $W'$ as defined in \cite[Chapitre~0, (4.1.6)]{EGAI}.
    This gives a morphism $\pi:\mathcal{X}\to T$ between smooth compact complex varieties.
    Moreover, setting 
    \begin{align}
        \mathcal{B} &\coloneqq \pi^{-1}(B),&
        \mathcal{D} &\coloneqq \overline{\pi^{-1}(D'\setminus B)},&
        \mathcal{W} &\coloneqq \overline{\pi^{-1}(W'\setminus B)},
    \end{align}
    and $f\coloneqq(\eta\circ\pi)\vert_{\mathcal{X}\setminus\mathcal{B}}$, principalization implies that $\mathcal{B}\cup \mathcal{D} \cup \mathcal{ W}$ is a simple normal crossing divisor in $\mathcal{X}$.
    Also note that $\mathcal{W}\subset \mathcal{D}$.
    Then $f$ induces an isomorphism over $X\setminus(R\cup W)$ and a modification of the pairs
    \begin{align}\label{eq:f_modification}
        f:(\mathcal{X}\setminus\mathcal{B},\mathcal{D}\setminus\mathcal{B})
        \to (X\setminus R,D\setminus R).
    \end{align}
    Lastly, set
    \begin{align}\label{eq:def_recusive_pair}
        \mathcal{D}_W \coloneqq \overline{f^{-1}(D\setminus(R\cup W))},\quad
        \mathcal{B}_D \coloneqq \mathcal{B}\cap\mathcal{D}_W,\quad
        \text{and } \mathcal{W}_D \coloneqq \mathcal{W}\cap\mathcal{D}_W.
    \end{align}
\end{construction}

This construction has the following properties.

\begin{lemma}[Compatible good compactifications]\label{lem:compatible_good_compactification}
    Assume that $X\setminus R$ is of dimension $n\geq 1$ and $D\setminus R$ is pure of dimension $n-1$.
    Let $W\subset D$ be closed and of dimension at most $n-2$, such that $X\setminus(R\cup W)$ and $D\setminus(R\cup W)$ are smooth.
    Then, the following hold:
    \begin{enumerate}
        \item $\mathcal{B}\cup\mathcal{D}\subset \mathcal{X}$ is a simple normal crossing divisor; $\mathcal{B}$ and $\mathcal{D}$ have no common components and the pair $(\mathcal{X},\mathcal{D})$ is a good compactification of $(\mathcal{X}\setminus\mathcal{B},\mathcal{D}\setminus\mathcal{B})$;
        \item $\mathcal{D}_W$ is smooth and the divisor $\mathcal{D}_W\cup\mathcal{W}=\mathcal{D}$ is simple normal crossing;
        \item The map $f$ restricts to a modification
        \begin{align}
            f_D:(\mathcal{D}_W\setminus\mathcal{B}_D,\mathcal{W}_D\setminus\mathcal{B}_D)
            \to(D\setminus R,W\setminus R);
        \end{align}
        \item $(\mathcal{D}_W,\mathcal{W}_D)$ is a good compactification of $(\mathcal{D}_W\setminus\mathcal{B}_D,\mathcal{W}_D\setminus\mathcal{B}_D)$.
    \end{enumerate}
\end{lemma}
\begin{proof}
    We start by showing $1$.
    This follows directly from the principalization theorem \cite{Hir64} since, $\mathcal{D}$ and $\mathcal{W}$ are the closures of the preimages of $D'$ and $W'$ away from $B$, whose ideal sheaves we simultaneously principalized.
    This shows in particular that $(\mathcal{X},\mathcal{D})$ is a good compactification of $(\mathcal{X}\setminus\mathcal{B},\mathcal{D}\setminus\mathcal{B})$ by Remark \ref{rem:snc_good_compactification}.

    \medskip

    We show $2$ next.
    By construction we have
    \begin{align}\label{eq:identify_calD}
        \mathcal D
        =\overline{f^{-1}(D\setminus R)}
        =\overline{
           f^{-1}(D\setminus(R\cup W))
           \cup f^{-1}(W\setminus R)}
        =\mathcal D_W\cup\mathcal W.
    \end{align}
    So $\mathcal{D}_W\cup \mathcal{W}$ is a simple normal crossing divisor because $\mathcal{D}$ is by $1.$
    Moreover, $\mathcal{D}_W$ and $\mathcal{W}$ do not share any component.
    Next, we show that $\mathcal{D}_W$ is smooth.
    Note that each of its irreducible components is smooth by $1.$, so the only singular points would be the intersections of the components.
    For the sake of contradiction, suppose two of the components in $\mathcal{D}_W$ intersect in a subvariety $C$.
    Then, $C$ must be of codimension two in $\mathcal{X}$ and it must also be contained in a component of $\mathcal{B}\cup\mathcal{W}$, as $\mathcal{D}_W\setminus(\mathcal{B}_D\cup\mathcal{W}_D)\cong D\setminus(W\cup R)$ was assumed smooth.
    Hence any intersection must lie over $W\cup R$.
    However, $\mathcal{B}\cup \mathcal{D}_W\cup \mathcal{W}$ is a simple normal crossing divisor so any such component $C$ lying in a component of $\mathcal{B}\cup \mathcal{W}$ must be of codimension $3$ in $\mathcal{X}$.
    Thus we arrive at a contradiction and $\mathcal{D}_W$ must be smooth.

    \medskip

     We next show $3$.
    The restriction $f_D$ is proper, since $f$ is proper and
    $\mathcal{D}_W\setminus\mathcal{B}_D$ is closed in
    $\mathcal{X}\setminus\mathcal{B}$.
    Moreover,
    \[
        f_D^{-1}(W\setminus R)=\mathcal{W}_D\setminus\mathcal{B}_D,
    \]
    by construction and $f_D$ is an isomorphism over $D\setminus(R\cup W)$ as $f$ was.
    Thus $f_D$ is a modification of pairs.

    \medskip

    Finally, we show $4.$
    Restricting the simple normal crossing divisor $\mathcal{B}\cup\mathcal{W}$ to the smooth divisor $\mathcal{D}_W$ yields a simple normal crossing divisor $\mathcal{B}_D\cup \mathcal{W}_D$ in $\mathcal{D}_W$.
    Suppose $\mathcal{B}_D\cup \mathcal{W}_D$ shared a common component $C$, then this would be codimension $2$ in $\mathcal{X}$.
    However $C$ would also be contained in $\mathcal{D}_W$ by construction, again a contradiction to $\mathcal{B}\cup\mathcal{D}_W\cup\mathcal{W}$ being a simple normal crossing divisor.
    Hence, $(\mathcal{D}_W,\mathcal{W}_D)$ is a good compactification.
\end{proof}

In dimension zero, we use the canonical map sending $\sum_p a_p[p]$ to the function $p\mapsto a_p$.
This applies to the boundary pair when $n=1$, in which case $W=\varnothing$.

\begin{lemma}\label{lem:open_canonical_map}
    Suppose $g(X\setminus R,D\setminus R)=0$ and $X\setminus(D\cup R)$ is smooth.
    Let $(\mathcal{X},\mathcal{D})$ be the good compactification of the modification $(\mathcal{X}\setminus\mathcal{B},\mathcal{D}\setminus\mathcal{B})$ of $(X\setminus R,D\setminus R)$ from Construction \ref{const:good_cmp}.
    Then, there is a well-defined canonical map
    \begin{align}
        \oOmega:H_n(X\setminus R,D\setminus R)
        &\to \olog^n(\mathcal{X}\setminus\mathcal{D})\\
        \sigma &\mapsto \Omega(\iota_*(f_*)^{-1}\sigma),
    \end{align}
    where $\iota:(\mathcal{X}\setminus\mathcal{B}, \mathcal{D}\setminus\mathcal{B}) \hookrightarrow(\mathcal{X},\mathcal{D})$ is the natural inclusion.
    Moreover, the restriction of the form $\oOmega(\sigma)$ to $X\setminus(R\cup D)$ is independent of the choice of good compactification.
\end{lemma}
\begin{proof}
    The strategy is to apply Brown and Dupont's construction of the canonical map $\Omega$ to the good compactification $(\mathcal{X},\mathcal{D})$ of the modification $(\mathcal{X}\setminus\mathcal{B},\mathcal{D}\setminus\mathcal{B})$ of $(X\setminus R, D\setminus R)$ in Construction \ref{const:good_cmp}.

    Start by recalling that the pushforward $f_*$ in \eqref{eq:f_modification} gives an isomorphism of mixed Hodge structures
    \begin{align}\label{eq:f*_iso}
        f_*:H_n(\mathcal{X}\setminus\mathcal{B},
                 \mathcal{D}\setminus\mathcal{B})
        \xlongrightarrow{\sim} H_n(X\setminus R,D\setminus R)
    \end{align}
    by \cite[Proposition 1.3]{brown-dupont}, where we used $f^{-1}(D\setminus R)=\mathcal{D}\setminus \mathcal{B}$, see \eqref{eq:identify_calD}.
    In particular, this also implies that $g(\mathcal{X}\setminus\mathcal{B},
                 \mathcal{D}\setminus\mathcal{B})
        =g(X\setminus R,D\setminus R)$.
    The inclusion of pairs $\iota:(\mathcal{X}\setminus\mathcal{B}, \mathcal{D}\setminus\mathcal{B}) \hookrightarrow(\mathcal{X},\mathcal{D})$ induces a map
    \begin{align}
        \iota_*:H_n(\mathcal{X}\setminus\mathcal{B}, \mathcal{D}\setminus\mathcal{B}) \to H_n(\mathcal{X},\mathcal{D}).
    \end{align}
    Then, combining the equality of genera induced by $f_*$ from above with \cite[Proposition 3.19]{brown-dupont} for the good compactification $(\mathcal{X},\mathcal{D})$ we get
    \[
        g(\mathcal{X},\mathcal{D}) =g(\mathcal{X}\setminus\mathcal{B}, \mathcal{D}\setminus\mathcal{B}) =g(X\setminus R,D\setminus R)=0.
    \]
    Since $\mathcal{X}\setminus\mathcal{D}$ is smooth, the canonical map $\Omega:H_n(\mathcal{X},\mathcal{D}) \to\olog^n(\mathcal{X}\setminus\mathcal{D})$ is well-defined.
    Therefore,
    \[
        \oOmega\coloneqq\Omega\circ\iota_*\circ(f_*)^{-1}: H_n(X\setminus R,D\setminus R) \to\olog^n(\mathcal{X}\setminus\mathcal{D})
    \]
    is well-defined.

    Lastly, note that any good compactification can be chosen to give an isomorphism \linebreak $\mathcal{X}\setminus (\mathcal{D} \cup \mathcal{B}) \cong X\setminus(D\cup R)$.
    Identify the common open subsets of the two models with $X\setminus(R\cup D)$ via their respective modifications.
    By naturality of $\Omega$ and of Poincar\'e duality, the restrictions of the two constructed forms represent the same cohomology class on $X\setminus(R\cup D)$, since both come from the same relative homology class $\sigma$, see \cite[\S 1.3]{brown-dupont}.
\end{proof}

\begin{lemma}\label{lem:open_forms}
    Recall the notation from Construction \ref{const:good_cmp}, and suppose that $X$ is normal. Let $\dsheaf_X$ be the dualizing sheaf on $X$.
    There is a natural injection
    \[
        r_*:\olog^n(\mathcal{X}\setminus\mathcal{D})
        \hookrightarrow H^0(X,\dsheaf_X(D)).
    \]
\end{lemma}
\begin{proof}
    Write $\bar f\coloneqq \eta\circ\pi:\mathcal{X}\to X$ for the composition of the maps in Construction \ref{const:good_cmp}. 
    Let $V$ be the complement of the open subset in $X$ over which $\bar f$ is an isomorphism.
    Since $X$ is normal, its singular locus has codimension at least $2$. Thus, $\eta$ is an isomorphism away from a set of codimension $2$. The same is true for the principalization map $\pi$, since the locus where $B\cup D' \cup W'$ is not simple normal crossing is also of codimension at least $2$.
    As the locus $V$ where $\bar f$ does not give an isomorphism is a closed subset of $X$, it must have codimension at least $2$.
    The restriction $r$ of $\bar f$ to $\bar f^{-1}(X\setminus V)$ induces the map $r_*:H^0(\mathcal{X},\omega_{\mathcal{X}}(\mathcal{D}))  \hookrightarrow H^0(X\setminus V,\dsheaf_X(D))$.
    It is well-defined since $\mathcal{D}$ is reduced and contained in $\bar f^{-1}(D)$, and it is injective since $\bar f^{-1}(X\setminus (D\cup R))$ is dense in $\mathcal{X}$.
    Therefore, the injection induced by the restriction $r=\bar f\vert_{\bar f^{-1}(X\setminus V)}$ and the isomorphism furnished by $\bar f$ give
    \begin{align}
        \olog^n(\mathcal{X}\setminus\mathcal{D})
        = H^0(\mathcal{X},\omega_{\mathcal{X}}(\mathcal{D}))
        \hookrightarrow H^0(X\setminus V,\dsheaf_X(D))
        \cong H^0(X,\dsheaf_X(D)),
    \end{align}
     where the last isomorphism follows from \cite[Proposition 1.6]{HartshorneReflexive} as $V$ has codimension at least $2$, and since $\dsheaf_X(D)$ is reflexive by construction.
\end{proof}

\begin{remark}\label{rem:everything_extends}
    By Lemmas \ref{lem:open_canonical_map} and \ref{lem:open_forms}, the forms in $\olog^n(\mathcal{X}\setminus\mathcal{D})$ determine rational differential top forms on $X$ which are regular as sections of $\dsheaf_X$ on $X\setminus D$ and have at most simple poles along $D$.
    In particular, they have no iterated poles over $R$.
    Their pullbacks to $\mathcal{X}$ have logarithmic poles only along $\mathcal{D}$.
    So they are candidates for canonical forms.
\end{remark}

Let $W\subset D$ be a closed subset.
Then, recall that (partial) boundary maps
\[
    \partial_\bullet : H_\bullet(X\setminus R, D\setminus R) \to H_{\bullet -1}(D\setminus R, W\setminus R)
\]
as in \cite[\S 1.1.3]{brown-dupont} are also well-defined for quasiprojective varieties.
Recall the notation from Lemma \ref{lem:compatible_good_compactification}.
Then, there is also a well-defined residue map
\[
    \res:\olog^n(\mathcal{X}\setminus \mathcal{D}) \to \olog^{n-1}(\mathcal{D}_W\setminus \mathcal{W}).
\]

\begin{proposition} \label{prop:residues}
    Recall the notation from Lemma \ref{lem:compatible_good_compactification}.
    Assume that $X\setminus R$ is of dimension $n\geq 1$ and $D\setminus R$ is pure of dimension $n-1$.
    Let $W\subset D$ be closed, of dimension at most $n-2$, such that $X\setminus(R\cup W)$ and $D\setminus(R\cup W)$ are smooth.
    Suppose that $g(X\setminus R,D\setminus R)=g(D\setminus R,W\setminus R)=0$.
    Then the following diagram commutes:
    \begin{equation}\label{eq:diagram}
        \begin{tikzcd}[column sep=1.8em,row sep=2.7em]
            H_n(X\setminus R,D\setminus R)
            \arrow[r,"\oOmega"]
            \arrow[d,"\partial"']
            & \olog^n(\mathcal{X}\setminus\mathcal{D})
            \arrow[d,"\res"]\\
            H_{n-1}(D\setminus R,W\setminus R)
            \arrow[r,"\oOmega"']
            & \olog^{n-1}(\mathcal{D}_W\setminus\mathcal{W}_D).
        \end{tikzcd}
    \end{equation}
\end{proposition}
\begin{proof}
    We show this by using the commutative property for the canonical map $\Omega$ for the good compactifications from Construction \ref{const:good_cmp}.
    Let
    \[
        \iota_D:
        (\mathcal{D}_W\setminus\mathcal{B}_D,
         \mathcal{W}_D\setminus\mathcal{B}_D)
        \hookrightarrow(\mathcal{D}_W,\mathcal{W}_D)
    \]
    be the restriction of the inclusion $\iota$, and let $\partial_{\mathcal{D}_W}$ be the partial boundary map associated with $\mathcal{D}=\mathcal{D}_W\cup\mathcal{W}$.
    By modification invariance and the invariance of genera for good compactifications, see \cite[Proposition 3.19]{brown-dupont}, both compact pairs below have genus zero.
    Moreover, by Lemma \ref{lem:compatible_good_compactification} both $\mathcal{X}$ and $\mathcal{D_W}$ are smooth, thus the canonical maps $\Omega$ are well-defined for these pairs.
    Consider
    \[
        \begin{tikzcd}[column sep=6.5em,row sep=3em]
            H_n(X\setminus R,D\setminus R)
            \arrow[r,"\iota_*\circ (f_*)^{-1}"]
            \arrow[d,"\partial"']
            & H_n(\mathcal{X},\mathcal{D})
            \arrow[r,"\Omega"]
            \arrow[d,"\partial_{\mathcal{D}_W}"']
            & \olog^n(\mathcal{X}\setminus\mathcal{D})
            \arrow[d,"\res"]\\
            H_{n-1}(D\setminus R,W\setminus R)
            \arrow[r,"(\iota_D)_*\circ ((f_D)_*)^{-1}"']
            & H_{n-1}(\mathcal{D}_W,\mathcal{W}_D)
            \arrow[r,"\Omega"']
            & \olog^{n-1}(\mathcal{D}_W\setminus\mathcal{W}_D).
        \end{tikzcd}
    \]
    The left square commutes by naturality of boundary maps
    and excision \cite[\S 1.1.3]{brown-dupont}, and the right
    square by the proof of \cite[Proposition 2.14]{brown-dupont}.
    Compositions of the horizontal maps are the maps $\oOmega$ for the two pairs, so the outer rectangle yields the claim.
\end{proof}

\subsection{The main theorem}

In what follows we write $X(\mathbb{R})$ for the real locus of a real variety $X$, and $[S]$ for the class of a semi-algebraic set $S$ in an appropriate relative homology group. 

\begin{lemma}\label{lem:partial-boundary-face}
    Let $X$ be a real variety of dimension $n$ and let $S$ be a compact regular semi-algebraic subset of $X$ with orientable interior.
    Let $F$ be a facet of $S$ with Zariski closure $W$, where $W$ is an irreducible component of $D=\partial_aS$.
    Let $D'$ be the union of all other irreducible components of $D$ and set $D_W \coloneqq W\cap D'$. Assume that $X$ is smooth in a neighborhood of any point in $\partial S\cap (W\setminus D_W)$. 
    Moreover, assume that $W\setminus D_W$ is smooth, and that the interior of $F$ is contained in it.
    Then we have $\partial_aF\subset D_W$ and
    \[
        \partial_W\bigl(\iota_*[S]\bigr) = \iota_*[F] \in H_{n-1}(W,D_W),
    \]
    where $\partial_W$ is the boundary map in relative homology, $\iota$ is the natural inclusion of the faces into their Zariski closures and $\iota_*$ is the pushforward in relative homology.
\end{lemma}

\begin{proof}
    Let $p\in\partial S\cap (W\setminus D_W)$.
    There is an analytic neighborhood $U\subset (X\setminus D')(\R)$ of $p$ such that $W\cap U$ is a smooth hypersurface in $U$.
    By assumption $X$ is smooth in a neighborhood of $p\in \partial S\cap (W\setminus D_W) $, hence we can choose local coordinates $x_1,\ldots, x_n$ such that $U\cap W$ is a coordinate hyperplane in $U$ given by the vanishing of $x_1$.
    The hypersurface $W$ divides $U$ into two connected components, which a priori are each either contained in $S$ or disjoint from $S$.
    The set $S$ must contain exactly one of these components because $p$ lies on the boundary of $S$. 
    Thus $U\cap S$ is homeomorphic to the half-ball $\{x\in \R^n\mid x_1 \geq 0\}$.
    This implies that the boundary of $S$ is locally at $p$ given by $W$, i.e.,
    \[
        \partial S\cap U = W\cap U.
    \]
    Since $W$ is locally given by the vanishing of $x_1$, it follows that $p\in \open F$. 
    By varying $p$, this shows $F\setminus D_W\subset \open F$.
    As we assumed the reverse inclusion, we get
    \[
        \open F = F\setminus D_W = \partial S\setminus D'.
    \]
    Combining this with $\open (F\cap U)\cong \R^{n-1}$ shows that $\open F$ is an orientable manifold, since it is contained in the half ball $\{x\in \R^n\mid x_1 \geq 0\}$, with orientation induced by $S$.
    Moreover, we get $\partial F\subset D_W$ and $\partial_a F\subset D_W$.

    \medskip
     
    Now we choose a simplicial triangulation of the compact set $S$ that is compatible with the semi-algebraic sets $\partial S$, $F$, and $S\cap D'$, which exists by \cite{hironaka1975triangulations}.
    We denote the simplicial chains in this compatible triangulation representing the fundamental classes of $S$ and $F$ by $c_S$ and $c_F$, respectively. Since our chosen orientation of $\open F$ is the one induced by $S$, and since $\partial S\setminus D' = \open F$, we get
    \begin{align}\label{eq:c_S}
        d_n(c_S) = c_F + c', 
    \end{align}
    where $d_n:C_n(S)\to C_{n-1}(S)$ is the simplicial differential and $c'$ is a chain supported on $D'\cap \partial S$.
    Next, we use the partial boundary map in \cite[Eq.~(10)]{brown-dupont}
    \begin{align}\label{eq:partial_boundary_map}
        \partial_F: H_n(S, F \cup (D'\cap \partial S)) \to H_{n-1}(F \cup (D'\cap \partial S), (D'\cap \partial S)) \cong H_{n-1}(F,\partial F),
    \end{align}
    where the last isomorphism follows from excision.
    By the above equation \eqref{eq:c_S} we get
    \[
        \partial_F([S])=[F] + [c'] \equiv [F] \in H_{n-1}(F,\partial F),
    \]
    where the class $[c']$ is $0$ in $H_{n-1}(F,\partial F)$ because $c'$ was supported on $D'\cap \partial S$.
    Composing with the pushforward induced by the inclusion $\iota\colon X(\R) \hookrightarrow X$, we get the desired identity of cycles under the map $H_n(X,W\cup D')\to H_{n-1}(W,D_W)$.
\end{proof}

Let us illustrate which situations are avoided by the assumptions of Lemma~\ref{lem:partial-boundary-face}.
\begin{example}
    Consider the union $\Delta\subset \R^2$ of the two regions $\Delta_1$ and $\Delta_2$ shown in Figure \ref{fig:twoSimplices}.
    Its algebraic boundary is given by the union of lines defined by $x$, $y$ and $x+y\pm 1$, where $x,y$ are the coordinates of $\R^2$.
    Consider the face $F$ of $\Delta$ given by restricting it to $y=0$, which is the line segment $[-1,1]\cong [-1,1]\times \{0\}$.
    However, $\open F = (-1,1)$ is not contained in $\cV(y) \setminus D_{y}=[-1,1]\setminus \{-1,0,1\}$, contradicting the assumptions in Lemma \ref{lem:partial-boundary-face}.
    Indeed, the partial boundary map $\partial_{y}:H_2(\P^2,\partial_a\Delta)\to H_{1}(\cV(y),D_{y})$ sends $[\Delta]$ to $[0,1]-[-1,0]$, which is not the class $[0,1]+[-1,0]$ of $[F]$.
\end{example}

     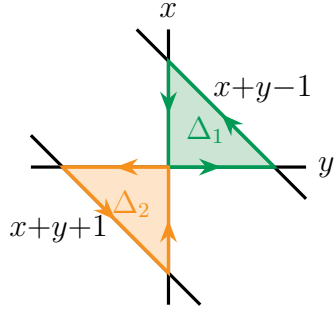
\begin{figure}[htb!]
         \centering
         \begin{tikzpicture}[scale=1.4]

        \draw[-, very thick] (-1.3,0) -- (1.3,0) node[right] {$y$};
        \draw[-, very thick] (0,-1.3) -- (0,1.3) node[above] {$x$};
        \draw[-, very thick] (-0.3,1.3) -- (1.3,-0.3) node[pos=0.5, right, above, xshift =0.55cm] {$x{+}y{-}1$};
        \draw[-, very thick] (-1.3,0.3) -- (0.3,-1.3) node[pos=0.7, left, above, xshift =-1.2cm] {$x{+}y{+}1$};
        
        \fill[ForestGreen!20] (0,0) -- (1,0) -- (0,1) -- cycle;
        \draw[ForestGreen, very thick] (0,0) -- (1,0) -- (0,1) -- cycle;
        \draw[ForestGreen, very thick, midarrow] (0,0) -- (1,0);
        \draw[ForestGreen, very thick, midarrow] (1,0) -- (0,1);
        \draw[ForestGreen, very thick, midarrow] (0,1) -- (0,0);
        
        \fill[BurntOrange!20] (0,0) -- (-1,0) -- (0,-1) -- cycle;
        \draw[BurntOrange, very thick] (0,0) -- (-1,0) -- (0,-1) -- cycle;
        \draw[BurntOrange, very thick, midarrow] (0,0) -- (-1,0);
        \draw[BurntOrange, very thick, midarrow] (-1,0) -- (0,-1);
        \draw[BurntOrange, very thick, midarrow] (0,-1) -- (0,0);
        
        \node[ForestGreen] at (0.35,0.35) {$\Delta_1$};
        \node[BurntOrange] at (-0.35,-0.35) {$\Delta_2$};
        
        \end{tikzpicture}
    \caption{Two triangles in the plane connected in codimension $2$.}
    \label{fig:twoSimplices}
     \end{figure}

We are now ready to formulate the main result of this section, namely Theorem \ref{thm:main_1} from the introduction.
Here we write $\oOmega$ for the canonical map from Lemma \ref{lem:open_canonical_map}, $[S]$ for the homology cycle in $H_n((X\setminus R)(\mathbb{R}),(D\setminus R)(\mathbb{R}))$ of the semi-algebraic set $S$, and ${\iota: ((X\setminus R)(\mathbb{R}),(D\setminus R)(\mathbb{R})) \hookrightarrow (X\setminus R,D\setminus R)}$ for the inclusion inducing the pushforward $\iota_*$ in relative homology. Finally, $r_*$ is the map induced by restriction from Lemma \ref{lem:open_forms}. 

\begin{theorem}\label{thm:fromBDtoPG} 
    Let $S$ be a non-empty regular semi-algebraic subset of a smooth real projective variety $X$ of dimension $n$ with orientable interior.
    Let $D$ be the algebraic boundary of $S$.
    Suppose $S$ has a well-defined face stratification $\cF(S)$ which is a graded poset of rank $n$. Let $R:=\mathcal{R}(S)$ be the residual arrangement of $S$.
    Further, assume for every face $F$ of $S$ of dimension $d>0$ with Zariski closure $V$ the following.
    \begin{enumerate}
        \item $V$ is normal and $V\setminus W$ is smooth, where $W$ is the union of all varieties in the algebraic stratification $\cL(D)$ that are codimension two subsets of $V$;

        \item $H^d(V,\mathcal{O}_V)=0$;

        \item The interior of $F$ is contained in $V\setminus D_V$, where $D_V$ is the union of all varieties in $\mathcal{L}(D)$ that are codimension one subsets of $V$.
    \end{enumerate}
    Suppose $(X\setminus R,D\setminus R)$ is a recursive genus zero pair, and the set $R$ is disjoint from $S$.
    Then $(X,S)$ is a positive geometry with canonical form  $r_*\oOmega(\iota_*[S])$.
\end{theorem}
\begin{proof}
    For simplicity we will write $(\mathring{X}, \mathring{D})$ for the pair $(X\setminus R, D\setminus R)$ and $\sigma$ for the homology class $[S]$. Since $(\mathring{X}, \mathring{D})$ is a recursive genus zero pair, and $X$ is smooth, the map 
    $$\oOmega: H_n(\mathring{X}, \mathring{D})\to \Omega^n_{\mathrm{log}}(\mathcal{X}\setminus\mathcal{D})$$
    exists by Lemma \ref{lem:open_canonical_map}.
    Here $(\mathcal{X},\mathcal{D})$ is a good compactification of a modification of $(\mathring{X},\mathring{D})$ as in Construction \ref{const:good_cmp}.
    By Lemma \ref{lem:open_forms} we get a rational differential $n$-form $r_*\oOmega(\sigma)$ on $X$, which is regular on $X\setminus D$, has logarithmic poles along $D$ and has no iterated poles along $R$. 

    We first show that $r_*\oOmega(\sigma)$ satisfies the recursive property.
    We prove the statement by recursion on the dimension of the faces.
    By assumption the interior of $S$ is orientable.
    Let $V$ be a facet stratum, that is, an irreducible component of $D=\partial_a S$, so that the corresponding facet $F$ is the closure of the interior of $V\cap \partial S$.
    So $F$ is non-empty, and it is regular by construction. Its interior is orientable by the proof of Lemma \ref{lem:partial-boundary-face}.
    Moreover, since $\mathcal{F}(S)$ is graded of rank $n$, the algebraic boundary $\partial_aF\subset D_V$ is pure of codimension one in $V$.
    
    Every pair $(V\setminus R,D_V\setminus R)$ for $V\in\cL(D)\setminus \mathcal{R}(S)$ has genus zero and the complement $V\setminus W$ is smooth by assumption. 
    Moreover, for every facet $F$ the Zariski closure $V$ must be smooth in a neighborhood of $\partial S\cap (V\setminus D_V)$, which can be seen as follows.
    Any codimension two stratum in $W$ is contained in the intersection of two distinct irreducible components of $D_V$: every codimension two element of $\mathcal{L}(D\setminus R)$ contained in $V$ must also lie in $D_V$, since $\mathcal{F}(S)$ is graded and the residual arrangement is removed.
    Hence, $W\cap V\subseteq D_V$ and thus $V\setminus D_V\subseteq V\setminus W$ where the latter is smooth by assumption.
    So $V$ is smooth at every point of $\partial S\cap(V\setminus D_V)$.
    Thus, Lemma~\ref{lem:partial-boundary-face} applies, and combining it with  Proposition~\ref{prop:residues} yields
    \[
      \res_V\bigl(\oOmega(\iota_*[S])\bigr) = \oOmega\bigl(\partial_V(\iota_*[S])\bigr) = \oOmega(\iota_*[F]).
    \]
    where $\mathrm{Res}_V$ is the iterated residue map along any chain in $\mathcal{F}(S)$ leading to $V$. 
    
    Now we can continue recursively by the assumption that the pair $(X\setminus R,D\setminus R)$ is a recursive genus zero pair, and because the face stratification of any face $F$ is a graded sub-poset of $\cF(S)$ of rank $\dim F$, since it is a principal lower order ideal.
    This shows that the iterated residue of $r_*\oOmega(\iota_*[S])$ has the correct logarithmic poles along the faces of $S$.
    By construction of the map $\oOmega$ (see Remark \ref{rem:everything_extends}), the form $r_*\oOmega(\iota_*[S])$ has vanishing iterated residues along the components of $\mathcal{R}(S)$, i.e. it has poles exactly along the faces of $S$. 

    We proceed by showing that the canonical forms appearing in the recursion have residues that evaluate to $\pm 1$ at the vertices. Firstly, by Remark \ref{rem:conrner_top_implies_corner_face} every face of $S$ has a vertex. 
    Now applying Proposition \ref{prop:residues} for $n=1$ (that is, to one-dimensional faces of $S$) together with the definition of the canonical map in dimension zero (immediately before Lemma \ref{lem:open_canonical_map}) we get that the residue of $\oOmega(\iota_*[S])$ at any vertex evaluates to $\pm 1$.
    This shows that $\oOmega(\iota_*[S])$ satisfies conditions (a), (b) and (c) from Definition \ref{def:pos_geom_lam}.

    It remains to prove that the canonical forms above are unique in each iterative step.
    Assume for the sake of contradiction that for the $d$-dimensional pair $(V, F)$ of a face $F\subset S$ with face stratum $V$, there exist two forms $\omega_1$ and $\omega_2$ satisfying conditions (b) and (c).
    Since $V$ is normal, we have that $\omega_1-\omega_2$ is a holomorphic $d$-form on $V$ and thus is equal to zero since $H^d(V, \cO_V)=0$. This can be shown as follows.
    Note that the dualizing sheaf $\dsheaf_{V}$ is isomorphic to the extension of the canonical line bundle $\Omega^d_{V}$ on the smooth locus of $V$ by \cite[Proposition 5.75]{birationalGeometry}.
    Therefore, by construction of the dualizing sheaf we have $0=H^d(V, \O_{V})^\vee \cong H^0(V, \dsheaf_{V})\cong H^0( V,\Omega^{[d]}_{V})$, where $\Omega^{[d]}\coloneq \det(\Omega^1_{V})^{\vee\vee}$. 
    Hence there are no non-zero holomorphic $d$-forms on $V$.
\end{proof}

\begin{example}
    Theorem \ref{thm:fromBDtoPG} allows one to recover the known results that polytopes in $\mathbb{P}^n$ and the totally non-negative Grassmannian $(\mathrm{Gr}(k,n), \mathrm{Gr}(k,n)_{\geq 0})$ are positive geometries. For simple polytopes all the assumptions of the theorem are trivially satisfied. For the totally non-negative Grassmannian, the necessary properties follow from \cite{postnikov} and \cite{Knuston2013positroid}.
\end{example}

\begin{example}\label{ex:nonstrong_pair}
    Let $X=\P^3$ and let $D$ be the union of the four hyperplanes
    $H_i\coloneqq\cV(\ell_i)$ and the smooth quadric $Q\coloneqq\cV(q)$,
    where
    \begin{align}
        \ell_1 &= 3772w + 824x + 533y - 1886z,\\
        \ell_2 &= 2506w - 14x - 468y - 1253z,\\
        \ell_3 &= 11430w - 1754x + 793y - 5715z,\\
        \ell_4 &= z,\\
        q &= 9w^2 - x^2 - y^2 - z^2.
    \end{align}
    Then $(X,D)$ is a recursive genus zero pair, because every element of $\cL(D)$ is smooth and rational, see, e.g.,~\cite[Proposition 3.26]{brown-dupont}.
    Note that for each $1\leq i<j\leq 3$, the intersection
    $H_i\cap H_j\cap Q$ consists of two rational points.
    Next, consider the semi-algebraic region
    \[
        S\coloneqq
        \left\{
            (x,y,z)\in\R^3 \mid 
            \ell_i(1,x,y,z)\geq 0\ \text{for }1\leq i\leq 4,
            \quad x^2+y^2+z^2\leq 9
        \right\} \subset \A^3\cong \P^3\setminus \cV(w).
    \]
    The region $S$ itself is shown in Figure~\ref{fig:nonstrong-region}, and its face stratification is shown in Figure~\ref{fig:nonstrong-poset}.
    It has five facets $F_1,F_2,F_3,F_4$ and $F_Q$, seven one-dimensional faces, of which $e_{ij}=H_i\cap H_j$ for $i<j\in [3]$ are linear, and the remaining ones $c_i$ are curve segments on $Q$.
    Lastly, $S$ has four vertices, namely $a=(0,0,2)=H_1\cap H_2\cap H_3$ and the three points $v_{ij}\in S\cap H_i\cap H_j\cap Q$, where $1\leq i<j\leq 3$.
    The facet $F_4=S\cap H_4$ is a disk whose boundary circle $c_4=F_{Q}\cap F_4$ contains no vertices. In Figure~\ref{fig:nonstrong-poset}, the \textcolor{teal}{teal}-colored part of the face stratification does not have any chain of maximal length $3$ in the face stratification $\cF(S)$.
    In particular, this implies that the face stratification $\cF(S)$ is not graded.
    
    \begin{figure}[htbp]
        \centering
        \begin{subfigure}[c]{0.4\linewidth}
            \centering
            \includegraphics[width=\linewidth]
                {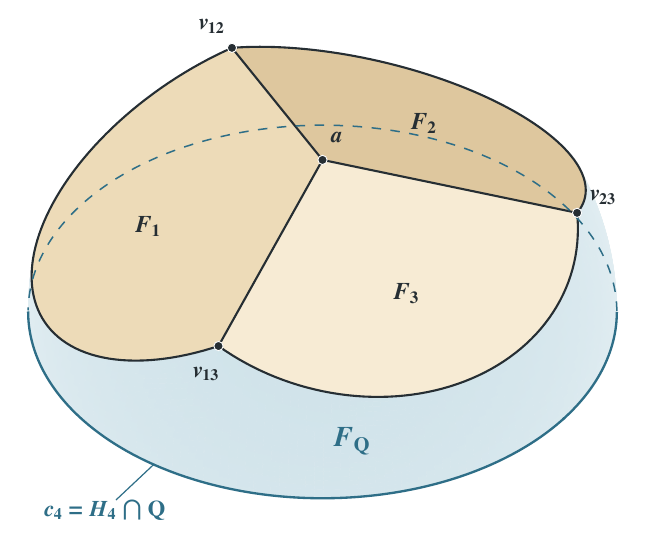}
            \caption{The region $S$.}
            \label{fig:nonstrong-region}
        \end{subfigure}
        \hfill
        \begin{subfigure}[c]{0.58\linewidth}
            \centering
            \includegraphics[width=\linewidth]
                {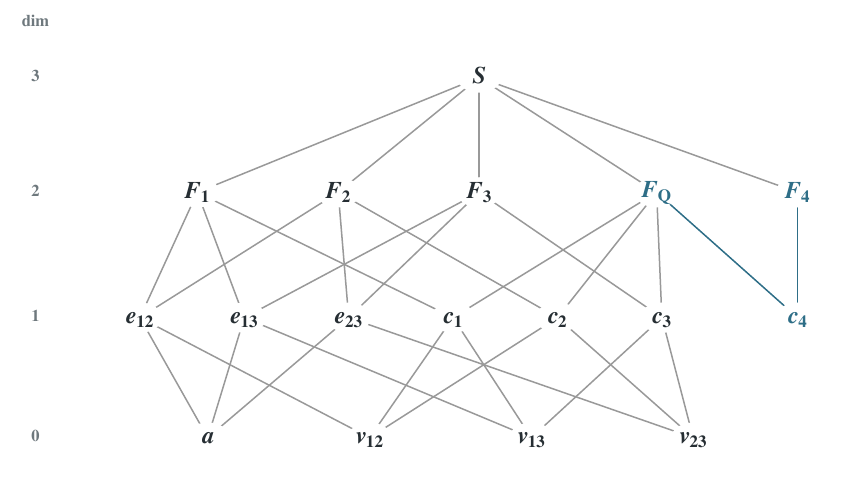}
            \caption{The face stratification $\cF(S)$.}
            \label{fig:nonstrong-poset}
        \end{subfigure}
        \caption{%
            The semi-algebraic set $S\subset \R^3$, which is not a positive geometry, bounded by a sphere and three planes, with facet $c_4=F_Q\cap H_4$ without vertices.%
        }
        \label{fig:nonstrong-geometry}
    \end{figure}
    
    Let us now discuss $S$ from the point of view of positive geometries.
    First, $S$ gives rise to a homology cycle in $H_3(X,D)$, which has three vertices on the quadric $Q$, one in each of the three intersections $H_i \cap H_j \cap Q$ with $i\neq j$, $i,j<4$. 
    So its residual arrangement is the union of the other $3$ points, one in each of the three intersections $H_i \cap H_j \cap Q$ with $i\neq j$, $i,j<4$, and three lines, namely $H_1 \cap H_4$, $H_2\cap H_4$, and $H_3\cap H_4$.
    The adjoint is therefore the product of the hyperplane $\cV(\ell_4)$ containing the three lines and the hyperplane spanned by the three points, which is defined by the linear form $\ell = 2360 x - 2145 y - 1036741 z + 2451652 w$.
    So the canonical form of $S$ is 
    \[
        \omega_S = \frac{\ell \cdot z}{\ell_1 \cdot \ell_2 \cdot \ell_3 \cdot q \cdot z} dx\wedge dy\wedge dz = \frac{\ell}{\ell_1 \cdot \ell_2 \cdot \ell_3 \cdot q} dx\wedge dy\wedge dz.
    \]
    This is not a canonical form of a positive geometry because it does not have a pole along the boundary component $H_4$.
    This is because the boundary curve $\cV(\ell_4,q) = H_4\cap Q$ does not contain a vertex of the cycle.
    In fact, from the point of view of the mixed Hodge theory framework, the reason is that the canonical form of the homology cycle in $H_3(X,D)$ defined by the semi-algebraic set $z\geq 0$, $9 \geq x^2 + y^2 + z^2$ is mapped to the zero-form: the canonical form is a global section of $\Omega^3_{\log}(\P^3\setminus (H_4\cup Q)) \cong  \Omega^3_{\P^3}(H_4\cup Q)$, where we used that $H_4\cup Q$ is a simple normal crossing divisor.
    However, $\Omega^3_{\P^3}(H_4\cup Q)\cong \mO_{\P^3}(-1)$, which does not have any global sections.
\end{example}

Our goal for the rest of the article is to apply Theorem \ref{thm:fromBDtoPG} to the $m=2$ amplituhedron $\mathcal{A}_{n,k,2}(Z)$.
This is not at all straightforward, as by \cite[Theorem 4.1]{KPScube} the pair $(\mathrm{Gr}(k,k+2), \partial_a \mathcal{A}_{n,k,2})$ has strictly positive genus for $k\geq 3$ and large enough $n$.
We will therefore produce a different pair which will be a recursive genus zero pair and in which the amplituhedron will give rise to a relative homology cycle.

\section{Faces of the $m=2$ amplituhedron} \label{sec:faces}

In this section we concentrate on the faces of the $m=2$ amplituhedron and for the sake of brevity write $\mathcal{A}_{n,k}:=\mathcal{A}_{n,k,2}(Z)$. 
The main goal of this section is to show that for generic choices of $Z$ the faces of the amplituhedron $\aA_{n,k}$ are supported on varieties of genus zero, and that any Zariski open subset of such varieties also has genus zero.
Moreover, as we show in \cite[Proposition 4.11]{KPScube}, the obvious obstructions for the pair $(\gr(k,k+2),\partial_a\aA_{n,k})$ to have vanishing genus come from residual strata $W\in \mathcal{L}(\partial_a\aA_{n,k})$. 
Therefore, we conclude that any obstruction for the pair $(\mathrm{Gr}(k,k+2), \partial_a\mathcal{A}_{n,k})$ to have genus zero must come from the residual arrangement $\mathcal{R}(\mathcal{A}_{n,k})$.
We will use this in the proof of our main Theorem \ref{thm:main_2} to bound the genus of the open pair $(\gr(k,k+2)\setminus R,\partial_a\aA_{n,k,2}(Z)\setminus R)$, where $R$ is the union of all residual components in the residual arrangement.

\medskip

This section is organized as follows.
We start in Subsection \ref{sec:ratsing_genus} by recording a number of technical results that are useful for computing genera of varieties with rational singularities.
Then, in Subsection \ref{sec:lam} we show that face strata of the amplituhedron are unirational varieties with precisely this type of singularities.
Moreover, we show that residual strata are disjoint from the amplituhedron,
which is one of the conditions of Theorem \ref{thm:fromBDtoPG}.
Lastly, in Subsection \ref{subsec:faces_genus_zero} we show that all face strata of the amplituhedron are genus zero varieties.

\subsection{Rational singularities, sheaf cohomology, and the genus}\label{sec:ratsing_genus}

In this subsection we present a number of technical results governing the behavior of the genus of a variety $X$ under resolutions of singularities and, more generally, surjective morphisms.
We will use these results in Subsection \ref{subsec:faces_genus_zero}, by applying them to the face strata of $\mathcal{A}_{n,k}$.

\medskip

In the following let $X$ be a complex projective variety, $\O_{X}$ its structure sheaf, and, if $X$ is normal, $\omega_X$ the canonical sheaf.
For a birational morphism $f:Y\to X$ of projective varieties also denote by $R^kf_*\cF$ the higher direct image sheaves on $X$ given by the sheafification of $U\mapsto H^k(f^{-1}(U),\cF)$, where $\cF$ is a coherent sheaf on $Y$.
We start with the definition of the type of singularities we will encounter when dealing with amplituhedra.

\begin{definition}\label{def:rational_sing}
    Let $X$ be a projective variety.
    $X$ has \emph{rational singularities} if for any resolution of singularities $\pi:\wt{X}\to X$ with $\wt X$ smooth we have:
    \begin{enumerate}
        \item $\pi_{*}\O_{\wt X}\cong \O_{X}$,
        \item the higher direct images vanish, i.e., $R^k\pi_{*}\O_{\wt X}=0$ for all $k>0$.
    \end{enumerate}
\end{definition}

\begin{remark}\label{rem:rational_sing_implies_normal}
    Note that varieties with rational singularities are normal by \cite[Remark~2.7]{SchwedeTakagi2008_rat_normal}.
    This essentially follows from the condition $\pi_{*}\O_{\wt X}\cong \O_X$.
\end{remark}

\begin{lemma}\label{dim_sum_hodge_numbers}
    Let $X$ be a $d$-dimensional irreducible complex projective variety.
    Then 
    \[
        g(X)+\mathrm{cr}(X) = \sum_{p\geq 0}h^{p,0}(X)\leq \dim H^d(X, \mathcal{O}_X).
    \]
    If $X$ has rational singularities, equality holds. 
\end{lemma}
\begin{proof}
    By construction of the Hodge filtration we have
    \[
        h^{0,q}(X) = \dim \grf^0\grw_qH^d(X)_\C=\dim \grw_q\grf^0H^d(X)_\C,
    \]
    see, e.g., \cite[\S 1.2.1]{Deligne1971}.
    Then, 
    \[
        \dim \grf^0H^d(X)_\C = \sum_{p=0,q\geq0} \dim H^d(X)_\C^{p,q}= \sum_{q\geq 0} \dim H^d(X)_\C^{0,q}.
    \]
    Now for complex projective varieties there is a surjection 
    \[
        H^d(X, \mathcal{O}_X) \twoheadrightarrow \mathrm{gr}^0_F H^d(X)_\C
    \]
    (see, e.g., \cite[Eq. (1.5)]{kovács2011intuitivedefinitiondubois}). Thus, by symmetry in the mixed Hodge numbers we have \linebreak $g(X)+\mathrm{cr}(X)=\dim \mathrm{gr}_F^0 H^d(X)_\C \leq \dim H^d(X, \mathcal{O}_X)$ proving the first part of the statement.  
    Now if $X$ has rational singularities, it follows from \cite[Theorem S]{rationaldubois} that $X$ has Du Bois singularities.
    In this case by \cite[Eq.~(1.3)]{kovács2011intuitivedefinitiondubois} the surjection above is an~isomorphism.
\end{proof}

We will need the following technical result, which we believe to be standard but have not found in the literature.  
\begin{lemma}\label{lem:surj}
    Let $X$ be a projective $d$-dimensional variety with rational singularities, and $f:Y\to X$ be a resolution of singularities, then we have $H^d(X, \mathcal{O}_X)\cong H^d(Y, \mathcal{O}_Y)$.
\end{lemma}
\begin{proof}
    Note that $R^qf_*\O_Y=0$ for all $q>0$ by Definition \ref{def:rational_sing} as $X$ has rational singularities.
    Therefore, the Leray spectral sequence degenerates on the second page
    \[
        E^{p,q}_2 = H^p(X,R^q f_*\O_{Y}) \Rightarrow H^{p+q}(Y,\O_{Y}),
    \]
    as all higher differentials $d_r^{p,q}:E_r^{p,q}\to E_r^{p+r,q+1-r}$ vanish, since only the row of the second page $E_2^{p,q}$ with $q=0$ is non-zero.
    This provides an isomorphism $H^d(Y,\mathcal{O}_{Y})\cong H^d(X, f_*\mathcal{O}_{Y})$.
    Lastly, by definition of rational singularities we have $f_*\O_{Y}\cong \O_X$, which yields the claim.
\end{proof}

\subsection{Faces and the residual arrangement in the twistor embedding}\label{sec:lam}

In this subsection we will show that face strata of the amplituhedron are varieties with rational singularities. This will allow us to apply the results of the preceding subsection in the proof of Proposition \ref{prop:faceshavegenuszero}, the main result of this section.

\medskip

The algebraic boundary stratification of $\mathcal{A}_{n,k}$ was completely described in \cite{lam_face_structure}, and we now recall the necessary results and notation. This description uses the notion of the \emph{twistor embedding}
\[
    T_Z:\gr(k,k+2)\hookrightarrow\gr(2,n),\quad X\mapsto  (\langle Xi_1i_2\rangle)_{1\leq i_1<i_2\leq n},
\]
where $X\in\grk$ is given as the rowspan of a $k\times(k+2)$ matrix, and $\langle Xi_1i_2 \rangle$ are twistor coordinates, obtained by stacking $X$ together with the rows $i_1$ and $i_2$ of the matrix $Z$ defining the amplituhedron map and taking the determinant of the resulting matrix.
By \cite[Lemma 1]{lam_face_structure}, the map $T_Z$ is indeed an embedding.

\medskip

The stratification $\mathcal{L}(\aA_{n,k})$ can be studied in the twistor embedding and by \cite{lam_face_structure} it is a subposet of the positroid stratification of $\gr(2,n)$.
That is, any stratum in $\mathcal{L}(\mathcal{A}_{n,k})$ is isomorphic to the intersection of a positroid variety $\Pi_{\mathcal{N}}$ and the image $T_Z(\gr(k,k+2))$. Denote the set of all rank 2 positroids on $n$ elements yielding face (not residual) strata of $\mathcal{A}_{n,k}$ by $P_{n,k}$. 
We start by recording the following consequence of these definitions.

\begin{proposition}\label{cor:face_and_boundary_poset_coincide}
    The face stratification $P_{n,k}$ as in \cite{lam_face_structure} is isomorphic to the face stratification $\cF(\aA_{n,k})$ as in Definition \ref{def:faces_residual}.
\end{proposition}
\begin{proof}
    By definition the face stratification $\cF(\aA_{n,k})$ is the sub-poset of the algebraic boundary stratification $\cL(\aA_{n,k})$ consisting of all strata which are not residual.
    These strata are precisely intersections of the positroid varieties $\Pi_{\mathcal{N}}$ for $\mathcal{N}\in P_{n,k}$ as defined in \cite{lam_face_structure} with $T_Z(\mathrm{Gr}(k,k+2))$.
    Hence, the claim follows from Theorem 25 in \emph{op. cit.}
\end{proof}

We now use this description to record the following topological property of the amplituhedron $\mathcal{A}_{n,k}$, which we will need in Section \ref{sec:last}.

\begin{proposition}\label{prop:residual_arrangement_disjoint}
For generic choices of the matrix $Z$ the residual arrangement $\mathcal{R}$ of $\mathcal{A}_{n,k}$ is disjoint from $\mathcal{A}_{n,k}$, i.e., $\mathcal{R}\cap \mathcal{A}_{n,k}=\varnothing$. In particular, $\mathcal{R}$ is disjoint from all faces of $\mathcal{A}_{n,k}$.
\end{proposition}

\begin{proof}
   By \cite[Proposition 38]{lam_face_structure}, any residual stratum $V\in \mathcal{R}(\aA_{n,k,2})$ is isomorphic, under the twistor embedding $T_Z$, to the intersection of $T_Z(\mathrm{Gr}(k,k+2))$ with a positroid variety $\Pi_{\mathcal{N}}$ with $\mathcal{N}\notin P_{n,k}$.
    By \cite[Theorem 23]{lam_face_structure} the only positroid varieties intersecting $\mathcal{A}_{n,k}$ are of the form $\Pi_{\mathcal{N}}$ for $\mathcal{N}\in P_{n,k}$. Thus, $V\cap \aA_{n,k,2}=\varnothing$.
\end{proof}

The description of $\partial_a \mathcal{A}_{n,k}$ from \cite{lam_face_structure} allows one to study the algebro-geometric properties of the elements of $\mathcal{L}(\mathcal{A}_{n,k})$. 
For instance, the proof of \cite[Theorem 25]{lam_face_structure} shows that for generic choices of the matrix $Z$ these varieties are normal. We now use this description to show that the open face strata are smooth and that the closed face strata are varieties with rational singularities. 
The smoothness of open strata is one of the conditions of Theorem \ref{thm:fromBDtoPG}, and rational singularities allow one to keep control of the genus under birational modifications, as explained in the previous subsection.

\begin{lemma}\label{lem:faces_have_rational_singularities}
    Let $Z$ be generic. For each $W\in\mathcal{L}(\mathcal{A}_{n,k})$ let $\mathring{W}$ be the complement of $W$ to all the elements of $\mathcal{L}(\mathcal{A}_{n,k})$ strictly contained in $W$. Then $\mathring{W}$ is smooth and $W$ has rational singularities. 
\end{lemma}
\begin{proof}
   For zero-dimensional strata there is nothing to show.
   So consider consider all strata of dimension at least one.
    By \cite[proof of Theorem~25 and Proposition~38]{lam_face_structure}, the twistor embedding identifies each closed stratum $W$ with an irreducible component of $ \Pi\cap\operatorname{Gr}(2,\operatorname{span}(Z))$, and the open stratum $\mathring{W}$ with the intersection of this component with $\mathring\Pi$ for a positroid variety $\Pi$ in $\mathrm{Gr}(2,n)$.
    Positroid varieties have rational singularities and their open counterparts are smooth \cite[Theorems~5.9 and~5.13]{Knuston2013positroid}.
    For generic $Z$, Kleiman transversality \cite[Theorem~2.5]{BilleyCoskun2012} and its rational-singularities analogue \cite[Theorem~2.5]{BCMP13} therefore imply that the closed intersections have rational singularities and the open intersections are smooth.
    Since there are finitely many positroid varieties, there is an Zariski open set of matrices $Z$ for which all conditions hold. The set of matrices $Z$ with positive maximal minors is Zariski dense in the set of all $n\times (k+m)$ matrices, so the claim~follows.
\end{proof}

\subsection{Faces have genus zero}\label{subsec:faces_genus_zero}

In this subsection we will show that every face stratum of $\mathcal{A}_{n,k}$ has genus zero.
Our strategy is as follows.
First, we will show that every face stratum is a unirational variety, which follows from the results of \cite{lam_face_structure}. Moreover, by Lemma \ref{lem:faces_have_rational_singularities} each face stratum has rational singularities. 
We then exploit the results from Subsection \ref{sec:ratsing_genus} to show that complex projective unirational varieties with rational singularities have genus zero. Finally, with a view towards Theorem \ref{thm:fromBDtoPG}, we show that any Zariski open subset of such a variety has genus zero. 

\begin{lemma} \label{lem:rationality}
    Every face stratum of $\mathcal{A}_{n,k}$ is a unirational variety. 
\end{lemma}

\begin{proof} 
It follows from the proof of \cite[Theorem 23]{lam_face_structure} that each face stratum of $\mathcal{A}_{n,k}$ is the closure of the image of a positroid variety under a rational map.
Positroid varieties admit an explicit rational parametrization, see, e.g., \cite{posparam}. Thus, every face stratum is also unirational.
\end{proof}

\begin{proposition}
\label{prop:faceshavegenuszero}
    A unirational complex projective variety $X$ with rational singularities has genus zero. If $d:=\dim X>0$, then moreover $H^d(X,\mathcal{O}_X)=0$. In particular, for a generic choice of the matrix $Z$ every face stratum $V$ of $\mathcal{A}_{n,k}$ has genus zero. If $d = \dim(V)>0$, then we also have $H^{d}(V,\cO_V) = 0$.
\end{proposition} 

\begin{proof}
By Lemma \ref{lem:rationality} every face stratum $V$ of $\mathcal{A}_{n,k}$ is a unirational variety. By Lemma \ref{lem:faces_have_rational_singularities} $V$ has rational singularities. Let $\wt{V}$ be a resolution of singularities of $V$ and $d:=\dim V$. We then get the following sequence of isomorphisms for $d>0$: 
$$0=H^d(\mathbb{P}^d, \mathcal{O}_{\mathbb{P}^d}) \cong H^d(\wt{V}, \mathcal{O}_{\wt{V}}) \cong H^d(V, \mathcal{O}_V).$$
The first isomorphism is \cite[Corollary 1.56]{KollarSmithCorti2004}.
The second isomorphism is Lemma~\ref{lem:surj}.
For $d>0$ the claim then follows from Lemma~\ref{dim_sum_hodge_numbers}.
If $d=0$, then $V$ is a point and has genus zero. Since we only used that $V$ is projective, unirational and has rational singularities, the claim holds for any complex variety with such properties.
\end{proof}

\begin{remark}
    Note that $V$ having rational singularities is essential for Proposition \ref{prop:faceshavegenuszero} to hold. In general, the genus of a variety can be higher than that of the resolution of singularities, see, e.g., \cite[Example 3.6]{brown-dupont}.
\end{remark}

\begin{lemma} \label{lem:genuszero_open}
    Let $X$ be an irreducible unirational complex projective variety of dimension $n>0$  with rational singularities.
    Then, for any Zariski open subvariety $U\subset X$ we have $\grf^0 H^n(U)=0$.
    In particular, its genus vanishes i.e. $g(U)=0$.
\end{lemma}
\begin{proof}
    We use an analog of Construction~\ref{const:good_cmp} with $R\coloneqq X\setminus U$, and $D\coloneqq\sing X$, where $D$ is now of codimension at least $2$ in $X$.
    Write $\bar f\coloneqq \eta\circ\pi:\mathcal{X}\to X$ for the composition  $\eta\circ \pi$, whose restriction to $\mathcal{X}\setminus\mathcal{B}$, with target $X\setminus R$, is the modification $f$ from Construction \ref{const:good_cmp}.

    First, note that
    \[
        \grf^0H_n(U)_\C
        \cong \grf^0H_n(U,D\setminus R)_\C,
    \]
    since $D\setminus R$ has codimension at least two in $U$,
    by \cite[proof of Proposition 3.9]{brown-dupont}.
    Moreover, $(\mathcal{X},\mathcal{D})$ is a good compactification
    of $(\mathcal{X}\setminus\mathcal{B}, \mathcal{D}\setminus\mathcal{B})$, which is a modification of $(U,D\setminus R)$.
    In particular, by \cite[Propositions 1.3 and 3.18]{brown-dupont}, we get
    \[
        \grf^0H_n(U)_\C \cong
        \grf^nH^n( \mathcal{X}\setminus\mathcal{D}, \mathcal{B}\setminus\mathcal{D})_\C.
    \]
    Next, consider the long exact sequence in relative cohomology
    \[
        \cdots
        \to H^{n-1}(\mathcal{B}\setminus\mathcal{D}) 
        \to H^n(\mathcal{X}\setminus\mathcal{D}, \mathcal{B}\setminus\mathcal{D})
        \to H^n(\mathcal{X}\setminus\mathcal{D})
        \to H^n(\mathcal{B}\setminus\mathcal{D})
        \to \cdots.
    \]
    After applying the exact functor $\grf^n$, the outermost terms
    vanish by the Hodge bounds and the inequality
    $\dim(\mathcal{B}\setminus\mathcal{D})<n$;
    see, for example,
    \cite[Théorème 8.2.4 (i), (ii)]{Deligne1974}
    or \cite[\S 1.2.2]{brown-dupont}.
    Therefore, we obtain an isomorphism
    \[
    \grf^nH^n(\mathcal{X}\setminus\mathcal{D}, \mathcal{B}\setminus\mathcal{D})_\C \cong \grf^nH^n(\mathcal{X}\setminus\mathcal{D})_\C.
    \]
    The latter is naturally identified with  $H^0(\mathcal{X},\omega_{\mathcal{X}}(\mathcal{D}))$
    by \cite[Eq.~(20)]{brown-dupont}, since $\mathcal{D}$ is a simple
    normal crossing divisor.
    
    We now show that $H^0(\mathcal{X},\omega_{\mathcal{X}}(\mathcal{D}))$ vanishes as follows.
    Since $\mathcal{D}$ is effective, we have an inclusion of sheaves $\omega_{\mathcal{X}} \into\omega_{\mathcal{X}}(\mathcal{D})$ on $\mathcal{X}$.
    By left-exactness of the pushforward, this gives an inclusion $\bar f_*\omega_{\mathcal{X}} \into \bar f_*\omega_{\mathcal{X}}(\mathcal{D})$ of sheaves on $X$.
    Write $j:X_{\mathrm{reg}}\hookrightarrow X$ for the open embedding of the regular locus.
    The restriction of $\bar f_*\omega_{\mathcal{X}}(\mathcal{D})$ to $X_{\mathrm{reg}}$ is $\omega_{X_{\mathrm{reg}}}$, because $\mathcal{D}$ is disjoint from $\bar f^{-1}(X_{\mathrm{reg}})$ and $\bar f_*\omega_{\mathcal{X}}=\dsheaf_X$ since $X$ has rational singularities.
    Therefore, we have
    \[
        j^*\bar f_*\omega_{\mathcal{X}}(\mathcal{D}) = j^*\bar f_*\omega_{\mathcal{X}} \cong \omega_{X_{\mathrm{reg}}}.
    \]
    This gives the inclusion
    \[
        \bar f_*\omega_{\mathcal{X}}(\mathcal{D}) \into j_*j^*\bar f_*\omega_{\mathcal{X}}(\mathcal{D}) \cong j_*\omega_{X_{\mathrm{reg}}},
    \]
    because $\bar f_*\omega_{\mathcal{X}}(\mathcal{D})$ is torsion-free.
    Thus we have established the following relations:
    \[
        \bar f_*\omega_{\mathcal{X}}
        \into \bar f_*\omega_{\mathcal{X}}(\mathcal{D})
        \into j_*\omega_{X_{\mathrm{reg}}}.
    \]
    Finally, the composition of both maps is the natural map, i.e., the unique map restricting to the identity on $X_{\mathrm{reg}}$.
    Both ends of this chain are equal to the dualizing sheaf of $X$.
    First, on the left, the assumption that $X$ has rational singularities implies $\bar f_*\omega_{\mathcal{X}}=\dsheaf_X$.
    Moreover, since $X$ is normal (see Remark~\ref{rem:rational_sing_implies_normal}), the dualizing sheaf of $X$ is equal to the pushforward $\dsheaf_X=j_*\omega_{X_{\mathrm{reg}}}$, under the inclusion of the canonical sheaf on the regular locus; see~\cite[Proposition 1.6]{HartshorneReflexive}, see also \cite[Eq. (5.3)]{Kovacs2013Stable}.
    So the composition of the two injections in the above equation gives an isomorphism, which implies that all sheaves on $X$ in the above equation are equal.
     Thus the global sections of $\omega_{\mathcal{X}}(\mathcal{D})$ and $\omega_{\mathcal{X}}$ must be the same, i.e. 
    \begin{align}
        H^0(\mathcal{X}, \omega_{\mathcal{X}}) = H^0(X, \bar f_*\omega_{\mathcal{X}})  = H^0(X, \bar f_*\omega_{\mathcal{X}}(\mathcal{D})) = H^0(\mathcal{X}, \omega_{\mathcal{X}}(\mathcal{D})).
    \end{align}
    Since $\bar f:\mathcal{X}\rightarrow X$ is birational, $\mathcal{X}$ is also unirational.
    Therefore,
    \[
        H^0(\mathcal{X},\omega_{\mathcal{X}})\hookrightarrow H^0(\PP^n,\omega_{\PP^n})=0,
    \]
    by \cite[Corollary 1.56]{KollarSmithCorti2004}.
    Thus we have
    \[
        0 = H^0(\mathcal{X},\omega_{\mathcal{X}})=H^0(\mathcal{X}, \omega_{\mathcal{X}}(\mathcal{D})) \cong \grf^nH^n(\mathcal{X}\setminus \mathcal{D})_\C\cong \grf^0 H_n(U)_\C,
    \]
    which after taking duals shows the first claim $\grf^0H^n(U)=0$.
    The remaining claim follows from the definition of the genus.
\end{proof}

\section{The $m=2$ amplituhedron is a positive geometry} \label{sec:last}
In this section we achieve our main goal and show that for generic choices of $Z$ the pair $(\gr(k,k+2), \mathcal{A}_{n,k,2}(Z))$ is a positive geometry. We start with the following auxiliary statement. It is a generalization of \cite[Corollary~3.14]{brown-dupont}, which covers the case of simple normal crossing divisors. It first appeared in the PhD thesis of the first author \cite{thesis}. 
Recall that $\mathcal{L}(D)$ denotes the algebraic stratification of a subvariety $D$ of a complex variety $X$. 

\begin{proposition}\label{prop:nonSNC_genus_bound}
    Let $X$ be a complex variety, and $D\subset X$ a codimension one subvariety.
    Suppose that $g(V)=0$ for all $V\in \cL(D)$.
    Then, the genus of the pair $(X,D)$ vanishes, i.e.~$g(X,D)=0.$
\end{proposition}
\begin{proof}
    We use the bound provided by \cite[Corollary 3.13]{brown-dupont} recursively by dimension. 
    We start in dimension $n=\dim(X)$ with $g(X,D)$, where this result gives the inequality
    $g(X,D) \leq g(X) + \sum_i g(D_i,D_i\cap \bigcup_{j\neq i}D_j)$.
    For an irreducible component $V\subset D_i\cap \bigcup_{j\neq i}D_j$ that has codimension at least two in $D_i$,  \cite[Proposition 3.9]{brown-dupont} tells us that $g(D_i,D_i\cap \bigcup_{j\neq i}D_j)$ equals $g(D_i,(D_i\cap \bigcup_{j\neq i}D_j)\setminus (V\setminus \bigcup_{j}D_i\cap D_j))$. 
    So we can replace the subset $D_i\cap \bigcup_{j\neq i} D_j$ of $D_i$ by the union of all its irreducible components $D_i^{(j)}$ of codimension one in $D_i$. 
    Then, \cite[Corollary 3.13]{brown-dupont} implies the inequality 
    \[ 
        g\left(D_i,\bigcup_{j}D_i^{(j)} \right) \leq g(D_i)+ \sum_{j}g\left(D_i^{(j)},D_i^{(j)}\cap \bigcup_{k\neq j}D_i^{(k)}\right). 
    \]
    Now, since the dimension of $D_i$ is $n-1$, we can apply the bound of  \cite[Corollary 3.13]{brown-dupont} to all new pairs $(D_i^{(j)},D_i^{(j)}\cap \bigcup_{k\neq j}D_i^{(k)})$, which eventually gives an inequality
    \[
        g(X,D) \leq \sum_{V\in \cL(D)} m_V g(V), 
    \]
    where $m_V$ are non-negative integers counting how often the same variety $V$ appears in the recursion.
    Since the algebraic stratification $\cL(D)$ is closed under taking intersections by construction, all $V$ that appear in the recursion are contained in $\cL(D)$.
    Hence, $g(V)=0$, and all terms in the sum above vanish, showing the claim.
\end{proof}

Therefore, we have an immediate consequence for recursive genus zero pairs.
\begin{corollary}\label{cor:all_genera_zero_implies_recursive}
    Let $X$ be a complex variety, and $D\subset X$ a codimension one subvariety.
    Suppose that $g(W)=0$ for all $W\in {\mathcal{L}}(D)$.
    Then $(X,D)$ is a recursive genus zero pair.
\end{corollary}
\begin{proof}
    Since ${\mathcal{L}}(D)$ is closed under intersections, Proposition \ref{prop:nonSNC_genus_bound} applies to any pair $(W,D_W)$, where $D_W$ is the union of all elements of $\mathcal{L}(D)$ that are codimension one subsets of~$W$.
\end{proof}

We are now ready to prove our main result. In its statement, the matrix $Z$ being generic means that it lies outside a proper Zariski-closed subset of $n\times (k+2)$ matrices intersected with the set of matrices with positive maximal minors.

\begin{theorem} \label{thm:ampli_pos_geo}
    Let $Z$ be a generic $n\times (k+2)$ matrix with positive maximal minors.
    Then $(\gr(k,k+2), \mathcal{A}_{n,k,2}(Z))$ is a positive geometry. 
\end{theorem}

\begin{proof}
    Let $R$ be the residual arrangement of $\mathcal{A}_{n,k,2}(Z)$ and consider the pair $(\mathring{X},\mathring{D}):=(\mathrm{Gr}(k,k+2)\setminus R,\partial_a \mathcal{A}_{n,k,2}(Z)\setminus R)$.
    By Proposition \ref{prop:residual_arrangement_disjoint} we have $\mathcal{A}_{n,k,2}(Z)\subseteq \mathring{X}$ and $\partial\aA_{n,k,2}(Z)\subset \mathring{D}$, so $\mathcal{A}_{n,k,2}(Z)$ gives rise to a well-defined relative cycle in $H_n(\mathring{X},\mathring{D})$. 
    We will now show that the assumptions of Theorem \ref{thm:fromBDtoPG} are satisfied for this pair and $S=\mathcal{A}_{n,k,2}(Z)$. 
    
    Our first step is to show that $(\mathring{X}, \mathring{D})$ is a recursive genus zero pair. 
    Since we have removed the residual arrangement, every element in the algebraic boundary stratification of $\mathcal{A}_{n,k,2}(Z)$ in $\mathring{X}$ is a Zariski open subset of a face stratum $V$ of $\mathcal{A}_{n,k,2}(Z)$. By Lemma \ref{lem:rationality} $V$ is a unirational variety, and by Lemma \ref{lem:faces_have_rational_singularities} it has rational singularities.  
    Thus by Lemma \ref{lem:genuszero_open} each  element in the boundary stratification of $\mathcal{A}_{n,k,2}(Z)$  in $\mathring{X}$ has vanishing genus.
    Then, $(\mathring{X},\mathring{D})$ is a recursive genus zero pair by Corollary \ref{cor:all_genera_zero_implies_recursive}. 

    Our second step is to show that faces of $\mathcal{A}_{n,k}$ satisfy the three conditions outlined in Theorem \ref{thm:fromBDtoPG}. 
    First, note that the amplituhedron $\mathcal{A}_{n,k}$ is a tame Grasstope \cite[Definition 2.2]{grasstopes}.
    Thus, by \cite[Proposition 2.3]{grasstopes} it is contained in the real points of an affine chart of $\mathrm{Gr}(k,k+2)$.
    The interior $\mathrm{int}(\mathcal{A}_{n,k})$ is thus homeomorphic to an open subset of a Euclidean space and is therefore orientable. 
    Now consider a face $F$ with Zariski closure $V$ in a closed subvariety of $\gr(k,k+2)$.
    By the proof of \cite[Theorem 25]{lam_face_structure} $V$ is normal and moreover its singularities are contained in the Zariski closures of the strata of $\mathcal{L}(D)$ that have codimension two in $V$.
    Both follow from Kleiman transversality together with \cite{Knuston2013positroid} and \cite[Theorem 2.7]{BilleyCoskun2012}, where these properties were shown for positroid varieties.
    Thus, Condition 1 of Theorem \ref{thm:fromBDtoPG} is satisfied. 
    Now, by Proposition \ref{prop:faceshavegenuszero} we have $H^d(V,\mathcal{O}_V)=0$ if $d:=\mathrm{dim}V>0$, implying Condition 2. 
    Finally, Condition 3 requires that the interior of $F$ be contained in $V\setminus D_V$. This again follows from \cite[Theorem 25]{lam_face_structure}.

    Note that $\aA_{n,k,2}(Z)$ has a graded face stratification of rank $2k$ by \cite[Theorem 25 and Theorem 33]{lam_face_structure}.
    Therefore, the semi-algebraic set $\mathcal{A}_{n,k}$ together with the pair $(\mathrm{Gr}(k,k+2)\setminus R, \partial_a \mathcal{A}_{n,k}\setminus R)$ satisfies the conditions of Theorem \ref{thm:fromBDtoPG}.
    Thus, $(\mathrm{Gr}(k,k+2), \mathcal{A}_{n,k})$ is a positive geometry with canonical form $r_*\oOmega(\iota_*[\aA_{n,k}])$, where $r_*$ is the map induced by restriction in Lemma \ref{lem:open_forms}, $\oOmega$ is the canonical map from Lemma \ref{lem:open_canonical_map}, and $\iota_*$ is the pushforward $H_{2k}(\mathring{X}(\mathbb{R}),\mathring{D}(\mathbb{R}))\to H_{2k}(\mathring{X},\mathring{D})$ induced by inclusion of pairs. 
\end{proof}

\begin{remark}
    Our proof does not give an explicit canonical form of $\mathcal{A}_{n,k,2}(Z)$. A natural way to construct it would be to sum the canonical forms of the amplituhedron tiles from \cite{parisi:meq2}. However, proving that such a sum is indeed the canonical form of the amplituhedron is a separate subtle technical task. 
\end{remark}

\bigskip
\noindent\textbf{AI Use.} We used AI (Claude Opus 5.5, Fable 5.1 and GPT-6 Astra Pro) as a search engine, for proofreading purposes of drafts of this work, and to generate the initial code for Figures \ref{fig:nonconvex-polygon} and \ref{fig:nonstrong-geometry}.
Parts of the proof of Lemma \ref{lem:partial-boundary-face} were suggested by AI and subsequently rewritten and streamlined by the authors.
The text of the article is entirely written by the authors, and they take full responsibility for its content, including all mathematical results.

\medskip
\noindent\textbf{Acknowledgements.} The authors are grateful to Daniele Agostini, Cl\'ement Dupont, Thomas Lam, and Christian Lehn for helpful discussions. This research is funded by the European Union (ERC, UNIVERSE PLUS, 101118787). Views and opinions expressed are, however, those of the authors only and do not necessarily reflect those of the European Union or the European Research Council Executive Agency. 
Neither the European Union nor the granting authority can be held responsible for them. 
The authors would like to thank the Erwin Schrödinger International Institute for Mathematics and Physics (ESI), University of Vienna (Austria), for the opportunity to participate in the Thematic Program “Amplitudes and Algebraic Geometry'' where an important part of this work has been accomplished.
RS was supported by the SPP 2458 ``Combinatorial Synergies'', funded by the Deutsche Forschungsgemeinschaft (DFG, German Research Foundation), project number 539677510. 
JK is grateful for the support received from the Institute for Advanced Study and an anonymous donor.

\printbibliography

@article{brown-dupont,
 author = {Brown, Francis and Dupont, Cl{\'e}ment},
 title = {Positive geometries and canonical forms via mixed {Hodge} theory},
 fjournal = {Communications in Mathematical Physics},
 journal = {Commun. Math. Phys.},
 volume = {406},
 number = {11},
 pages = {65},
 note = {Id/No 267},
 year = {2025},
 doi = {10.1007/s00220-025-05399-y},
 zbMATH = {8108277}
}

@article{Sinn2015,
 author = {Sinn, Rainer},
 title = {Algebraic boundaries of convex semi-algebraic sets},
 fjournal = {Research in the Mathematical Sciences},
 journal = {Res. Math. Sci.},
 volume = {2},
 pages = {18},
 note = {Id/No 3},
 year = {2015},
 doi = {10.1186/s40687-015-0022-0},
 zbMATH = {6587775},
 Zbl = {1381.52007}
}

@article{Deligne1974,
 author = {Deligne, Pierre},
 title = {Th{\'e}orie de {Hodge}. {III}},
 fjournal = {Publications Math{\'e}matiques},
 journal = {Publ. Math., Inst. Hautes {\'E}tud. Sci.},
 volume = {44},
 pages = {5--77},
 year = {1974},
 language = {French},
 doi = {10.1007/BF02685881},
 url = {https://eudml.org/doc/103935},
 zbMATH = {3375638},
 Zbl = {0237.14003}
}

@article{lam_face_structure,
 author = {Lam, Thomas},
 title = {On the face stratification of the {{\(m = 2\)}} amplituhedron},
 fjournal = {Combinatorial Theory},
 journal = {Comb. Theory},
 volume = {5},
 number = {4},
 pages = {17},
 note = {Id/No 3},
 year = {2025},
 doi = {10.5070/C65465669},
 zbMATH = {8169889}
}

@article{parisi:meq2,
 author = {Parisi, Matteo and Sherman-Bennett, Melissa and Williams, Lauren K.},
 title = {The {{\(m=2\)}} amplituhedron and the hypersimplex: signs, clusters, tilings, {Eulerian} numbers},
 fjournal = {Communications of the American Mathematical Society},
 journal = {Commun. Am. Math. Soc.},
 volume = {3},
 pages = {329--399},
 year = {2023},
 doi = {10.1090/cams/23},
 zbMATH = {7751004},
 Zbl = {1531.05246}
}

@article{pos_geom_2017,
 author = {Arkani-Hamed, Nima and Bai, Yuntao and Lam, Thomas},
 title = {Positive geometries and canonical forms},
 fjournal = {Journal of High Energy Physics},
 journal = {J. High Energy Phys.},
 volume = {2017},
 number = {11},
 pages = {124},
 note = {Id/No 39},
 year = {2017},
 doi = {10.1007/JHEP11(2017)039},
 zbMATH = {6854036},
 Zbl = {1383.81273}
}

@article{Knuston2013positroid,
 author = {Knutson, Allen and Lam, Thomas and Speyer, David E.},
 title = {Positroid varieties: juggling and geometry},
 fjournal = {Compositio Mathematica},
 journal = {Compos. Math.},
 year = {2013},
 volume = {149},
 number = {10},
 pages = {1710--1752},
 doi = {10.1112/S0010437X13007240},
 zbMATH = {6231006},
 Zbl = {1330.14086}
}

@article{Amplituhedron,
 author = {Arkani-Hamed, Nima and Trnka, Jaroslav},
 title = {The amplituhedron},
 fjournal = {Journal of High Energy Physics},
 journal = {J. High Energy Phys.},
 volume = {2014},
 number = {10},
 pages = {33},
 note = {Id/No 030},
 year = {2014},
 doi = {10.1007/JHEP10(2014)030},
 zbMATH = {7389633},
 Zbl = {1468.81075}
}

@article{Deligne1971,
 author = {Deligne, Pierre},
 title = {Th{\'e}orie de {Hodge}. {II}. ({Hodge} theory. {II})},
 fjournal = {Publications Math{\'e}matiques},
 journal = {Publ. Math., Inst. Hautes {\'E}tud. Sci.},
 volume = {40},
 pages = {5--57},
 year = {1971},
 language = {French},
 doi = {10.1007/BF02684692},
 url = {https://eudml.org/doc/103914},
 zbMATH = {3348292},
 Zbl = {0219.14007}
}

@misc{bao2019tilings,
  author        = {Bao, Huanchen and He, Xuhua},
  title         = {The $m=2$ amplituhedron},
  year          = {2019},
  eprint        = {1909.06015},
  archivePrefix = {arXiv},
  primaryClass  = {math.CO}
}

@incollection{kovács2011intuitivedefinitiondubois,
  author        = {Kov{\'a}cs, S{\'a}ndor J.},
  title         = {The intuitive definition of {D}u {B}ois singularities},
  booktitle     = {Geometry and Arithmetic},
  editor        = {Faber, Carel and Farkas, Gavril and de Jong, Robin},
  series        = {EMS Series of Congress Reports},
  publisher     = {European Mathematical Society},
  address       = {Z{\"u}rich},
  year          = {2012},
  pages         = {257--266},
  eprint        = {1109.5569},
  archivePrefix = {arXiv},
  primaryClass  = {math.AG}
}

@phdthesis{thesis,
    author = {Joris Koefler},
    title = {Positive Geometries in Physics},
    school = {University of G\"ottingen},
    year = {2026}
}

@incollection{williams2023positive,
 author = {Williams, Lauren K.},
 title = {The positive {Grassmannian}, the amplituhedron, and cluster algebras},
 year = {2023},
 publisher = {Berlin: European Mathematical Society (EMS)},
 doi = {10.4171/ICM2022/126},
 zbMATH = {7821709},
 Zbl = {1555.14105}
}

@article{amplituhedron_adjoint,
   title={Adjoints and canonical forms of tree amplituhedra},
   DOI={10.7146/math.scand.a-149816},
   fjournal = {Mathematica Scandinavica},
   journal={Math. Scand.},
   publisher={Det Kgl. Bibliotek/Royal Danish Library},
   author={Ranestad, Kristian and Sinn, Rainer and Telen, Simon},
 volume = {130},
 number = {3},
 pages = {433--466},
   year={2024}
}

@article{lam2024invitation,
  title={An invitation to positive geometries},
  author={Lam, Thomas},
  journal={Open problems in algebraic combinatorics},
  year={2024},
  publisher={American Mathematical Society}
}

@misc{postnikov,
      title={Total positivity, {Grassmannians}, and networks}, 
      author={Alexander Postnikov},
      year={2006},
      eprint={math/0609764},
      archivePrefix={arXiv},
      primaryClass={math.CO},
}

@misc{sturmfels2026positivegeometriescubicsurfaces,
      title={Positive Geometries from Cubic Surfaces}, 
      author={Bernd Sturmfels and Simon Telen},
      year={2026},
      eprint={2605.11909},
      archivePrefix={arXiv},
      primaryClass={math.AG},
}

@misc{KPScube,
  author        = {Koefler, Joris and Pavlov, Dmitrii and Sinn, Rainer},
  title         = {Positive Genus Pairs from Amplituhedra},
  year          = {2026},
  eprint        = {2601.11142},
  archivePrefix = {arXiv},
  primaryClass  = {math.AG}
}

@article{rationaldubois,
 author = {Kov{\'a}cs, S{\'a}ndor J.},
 title = {Rational, log canonical, {Du} {Bois} singularities: {On} the conjectures of {Koll{\'a}r} and {Steenbrink}},
 fjournal = {Compositio Mathematica},
 journal = {Compos. Math.},
 volume = {118},
 number = {2},
 pages = {123--133},
 year = {1999},
 doi = {10.1023/A:1001120909269},
 zbMATH = {1347315},
 Zbl = {0962.14011}
}

@article{lukowski,
 author = {{\L}ukowski, Tomasz},
 title = {On the boundaries of the {{\(m = 2\)}} amplituhedron},
 fjournal = {Annales de l'Institut Henri Poincar{\'e} D. Combinatorics, Physics and their Interactions},
 journal = {Ann. Inst. Henri Poincar{\'e} D, Comb. Phys. Interact.},
 volume = {9},
 number = {3},
 pages = {525--541},
 year = {2022},
 doi = {10.4171/AIHPD/124},
 zbMATH = {7665074},
 Zbl = {1520.81129}
}

@article{SchwedeTakagi2008_rat_normal,
  author       = {Schwede, Karl and Takagi, Shunsuke},
  title        = {Rational singularities associated to pairs},
  fjournal = {Michigan Mathematical Journal},
  journal = {Mich. Math. J.},
  volume       = {57},
  date         = {2008},
  pages        = {625--658},
  doi          = {10.1307/mmj/1220879429}
}

@misc{logcanonical,
  author        = {Hollering, Benjamin and Pavlov, Dmitrii and Pratt, Elizabeth},
  title         = {Log Canonical Models and Positive Geometries},
  year          = {2026},
  eprint        = {2607.28368},
  archivePrefix = {arXiv},
  primaryClass  = {math.AG}
}

@inproceedings{hironaka1975triangulations,
  author    = {Hironaka, Heisuke},
  title     = {Triangulations of algebraic sets},
  editor    = {Hartshorne, Robin},
  booktitle = {Algebraic Geometry, Arcata 1974},
  series    = {Proceedings of Symposia in Pure Mathematics},
  volume    = {29},
  date      = {1975},
  pages     = {165--185},
  publisher = {American Mathematical Society},
  location  = {Providence, RI},
  doi       = {10.1090/pspum/029/0374131}
}

@article{karpwilliams,
 author = {Karp, Steven N. and Williams, Lauren K.},
 title = {The {{\(m=1\)}} amplituhedron and cyclic hyperplane arrangements},
 fjournal = {IMRN. International Mathematics Research Notices},
 journal = {Int. Math. Res. Not.},
 volume = {2019},
 number = {5},
 pages = {1401--1462},
 year = {2019},
 doi = {10.1093/imrn/rnx140},
 zbMATH = {7130878},
 Zbl = {1429.52024}
}

@article{bernd1988,
 author = {Sturmfels, Bernd},
 title = {Totally positive matrices and cyclic polytopes},
 fjournal = {Linear Algebra and its Applications},
 journal = {Linear Algebra Appl.},
 volume = {107},
 pages = {275--281},
 year = {1988},
 doi = {10.1016/0024-3795(88)90250-9},
 zbMATH = {4065182},
 Zbl = {0653.15014}
}

@article{BCFW2025,
  author        = {Even-Zohar, Chaim and Lakrec, Tsviqa and Tessler, Ran J.},
  title         = {The amplituhedron {BCFW} triangulation},
  fjournal = {Inventiones Mathematicae},
  journal       = {Invent. Math.},
  volume        = {239},
  number        = {3},
  year          = {2025},
  pages         = {1009--1138},
}

@misc{m4tilings2023,
  author        = {Even-Zohar, Chaim and Lakrec, Tsviqa and Parisi, Matteo and Sherman-Bennett, Melissa and Tessler, Ran and Williams, Lauren},
  title         = {Cluster algebras and tilings for the $m=4$ amplituhedron},
  year          = {2023},
  eprint        = {2310.17727},
  archivePrefix = {arXiv},
  primaryClass  = {math.CO}
}

@article{whatisposgeo,
 author = {Ranestad, Kristian and Sturmfels, Bernd and Telen, Simon},
 title = {What is positive geometry?},
 fjournal = {Le Matematiche},
 journal = {Matematiche},
 issn = {0373-3505},
 volume = {80},
 number = {1},
 pages = {3--16},
 year = {2025},
 doi = {10.4418/2025.80.1.1},
 zbMATH = {8097746},
 Zbl = {1573.83147}
}

@article{LukowskiParisiWilliams2023,
  author        = {{\L}ukowski, Tomasz and Parisi, Matteo and Williams, Lauren K.},
  title         = {The positive tropical {G}rassmannian, the hypersimplex, and the $m=2$ amplituhedron},
  fjournal = {International Mathematics Research Notices},
  journal       = {Int. Math. Res. Not.},
  volume        = {2023},
  number        = {19},
  year          = {2023},
  pages         = {16778--16836},
  doi           = {10.1093/imrn/rnad010}
}

@article{GalashinLam2020,
  author        = {Galashin, Pavel and Lam, Thomas},
  title         = {Parity duality for the amplituhedron},
  fjournal = {Compositio Mathematica},
  journal       = {Compos. Math.},
  volume        = {156},
  number        = {11},
  year          = {2020},
  pages         = {2207--2262},
  doi           = {10.1112/S0010437X20007411}
}

@misc{Galashin2024origamiI,
  author        = {Galashin, Pavel},
  title         = {Amplituhedra and origami, {I}: tree level},
  year          = {2024},
  eprint        = {2410.09574},
  archivePrefix = {arXiv},
  primaryClass  = {hep-th}
}

@misc{Galashin2026origamiII,
  author        = {Galashin, Pavel},
  title         = {Amplituhedra and origami, {II}: loop level},
  year          = {2026},
  eprint        = {2606.03439},
  archivePrefix = {arXiv},
  primaryClass = {hep-th}
}

@book{birationalGeometry,
  author    = {Koll{\'a}r, J{\'a}nos and Mori, Shigefumi},
  title     = {Birational Geometry of Algebraic Varieties},
  series    = {Cambridge Tracts in Mathematics},
  number    = {134},
  publisher = {Cambridge University Press},
  location  = {Cambridge},
  date      = {1998},
  doi       = {10.1017/CBO9780511662560},
  isbn      = {9780521632775}
}

@article{psw,
 author = {Postnikov, Alexander and Speyer, David and Williams, Lauren},
 title = {Matching polytopes, toric geometry, and the totally non-negative {Grassmannian}.},
 fjournal = {Journal of Algebraic Combinatorics},
 journal = {J. Algebr. Comb.},
 issn = {0925-9899},
 volume = {30},
 number = {2},
 pages = {173--191},
 year = {2009},
 doi = {10.1007/s10801-008-0160-1},
 zbMATH = {5600049},
 Zbl = {1264.20045}
}

@article{posparam,
 author = {Karpman, Rachel},
 title = {Bridge graphs and {Deodhar} parametrizations for positroid varieties},
 fjournal = {Journal of Combinatorial Theory. Series A},
 journal = {J. Comb. Theory, Ser. A},
 issn = {0097-3165},
 volume = {142},
 pages = {113--146},
 year = {2016},
 language = {English},
 doi = {10.1016/j.jcta.2016.03.005},
 zbMATH = {6582001},
 Zbl = {1337.05114}
}

@article{HartshorneReflexive,
  author       = {Hartshorne, Robin},
  title        = {Stable reflexive sheaves},
  fjournal = {Mathematische Annalen},
  journal  = {Math. Ann.},
  date         = {1980},
  volume       = {254},
  number       = {2},
  pages        = {121--176},
  doi          = {10.1007/BF01467074},
}

@article{Hir64,
    author       = {Hironaka, Heisuke},
    title        = {Resolution of Singularities of an Algebraic Variety
                    Over a Field of Characteristic Zero. {I}, {II}},
    fjournal = {Annals of Mathematics},
    journal= {Ann. Math.},
    series       = {2},
    date         = {1964},
    volume       = {79},
    pages        = {109--203, 205--326}
}

@article{Wlo05,
    author       = {W{\l}odarczyk, Jaros{\l}aw},
    title        = {Simple {Hironaka} Resolution in Characteristic Zero},
    fjournal = {Journal of the American Mathematical Society},
    journal = {J. Am. Math. Soc.},
    date         = {2005},
    volume       = {18},
    number       = {4},
    pages        = {779--822},
    doi          = {10.1090/S0894-0347-05-00493-5},
    eprint       = {math/0401401},
    eprinttype   = {arxiv}
}

@book{KollarSmithCorti2004,
  author    = {Kollár, János and Smith, Karen E. and Corti, Alessio},
  title     = {Rational and Nearly Rational Varieties},
  series    = {Cambridge Studies in Advanced Mathematics},
  number    = {92},
  publisher = {Cambridge University Press},
  location  = {Cambridge},
  date      = {2004},
  doi       = {10.1017/CBO9780511734991},
}

@incollection{Kovacs2013Stable,
    author    = {Kov{\'a}cs, S{\'a}ndor J.},
    title     = {Singularities of stable varieties},
    booktitle = {Handbook of Moduli. Volume {II}},
    editor    = {Farkas, Gavril and Morrison, Ian},
    series    = {Advanced Lectures in Mathematics},
    number    = {25},
    pages     = {159--203},
    publisher = {International Press},
    location  = {Somerville, MA},
    date      = {2013},
    eprint    = {1102.1240},
    eprinttype = {arxiv},
    eprintclass = {math.AG}
}

@article{EGAI,
    author       = {Grothendieck, Alexander},
    title        = {Éléments de géométrie algébrique: {I}. {Le} langage des schémas},
    fjournal = {Publications Mathématiques de l'IHÉS},
    journal = {Publ. Math., Inst. Hautes {\'E}tud. Sci.},
    date         = {1960},
    volume       = {4},
    pages        = {5--228},
    doi          = {10.1007/BF02684778},
    langid       = {french}
}

@article{BilleyCoskun2012,
    author       = {Billey, Sara and Coskun, Izzet},
    title        = {Singularities of generalized {Richardson} varieties},
    fjournal = {Communications in Algebra},
    journal = {Commun. Algebra},
    date         = {2012},
    volume       = {40},
    number       = {4},
    pages        = {1466--1495},
    doi          = {10.1080/00927872.2011.551903},
    eprint       = {1008.2785},
    eprinttype   = {arxiv},
    eprintclass  = {math.AG}
}

@misc{anaelle,
 author = {Ana{\"e}lle Pfister},
 title = {Affine hyperplane arrangements at finite distance},
 year = {2026},
 eprint = {2602.00847},
 archivePrefix={arXiv},
 primaryClass = {math.AG}
}

@article{grasstopes,
 author = {Mandelshtam, Yelena and Pavlov, Dmitrii and Pratt, Elizabeth},
 title = {Combinatorics of {{\(m = 1\)}} {Grasstopes}},
 fjournal = {Combinatorial Theory},
 journal = {Comb. Theory},
 volume = {5},
 number = {2},
 pages = {21},
 year = {2025},
 doi = {10.5070/C65265411},
 zbMATH = {8121606},
 Zbl = {1573.05393}
}

@article{BCMP13,
  author       = {Buch, Anders S. and Chaput, Pierre-Emmanuel
                  and Mihalcea, Leonardo C. and Perrin, Nicolas},
  title        = {Finiteness of cominuscule quantum {$K$}-theory},
  journaltitle = {Annales scientifiques de l'École normale supérieure},
  series       = {4},
  volume       = {46},
  number       = {3},
  date         = {2013},
  pages        = {477--494},
  doi          = {10.24033/asens.2194},
  eprint       = {1011.6658},
  eprinttype   = {arxiv},
  eprintclass  = {math.AG},
}

\noindent{\bf Authors' addresses:}
\medskip

\noindent Joris Koefler, Institute for Advanced Study, Princeton

\noindent E-Mail: {\tt koefler@ias.edu}
\medskip

\noindent Dmitrii Pavlov, Max Planck Institute for Physics, Garching bei M\"unchen

\noindent
E-Mail: {\tt pavlov@mpp.mpg.de}

\medskip
\noindent Rainer Sinn, Georg-August-Universit\"at G\"ottingen

\noindent
E-Mail: {\tt rainer.sinn@mathematik.uni-goettingen.de}

\end{document}